\documentclass[12pt]{amsart}
\usepackage{amssymb, latexsym, amsmath, amsfonts}
\usepackage{verbatim}
\newtheorem{thm}{Theorem}[section]
\newtheorem{cor}[thm]{Corollary}
\newtheorem{lem}[thm]{Lemma} 

\theoremstyle{definition} 

\theoremstyle{remark}
\newtheorem{rem}[thm]{Remark}
\numberwithin{equation}{section}

\newcommand{\ov}{\overline}

\newcommand{\Om}{\Omega}

\newcommand{\al}{\alpha}

\begin{document}
\author{G. P. Balakumar}
\address{Department of Mathematics,
Indian Institute of Science, Bangalore 560 012, India}
\email{gpbalakumar@math.iisc.ernet.in}
\title{Model domains 
in $\mathbb{C}^3$ with abelian automorphism group}
\begin{abstract}
It is shown that every hyperbolic rigid polynomial domain in $\mathbb{C}^3$ of finite type, with abelian automorphism group is equivalent to a domain that is balanced with respect to some weight. 
\end{abstract}
\maketitle

\section{Introduction}
\noindent Let $D$ be a Kobayashi hyperbolic domain in $\mathbb{C}^n$. Then it is known that the group of holomorphic automorphisms of $D$, denoted ${\rm Aut}(D)$ is a real Lie group in the compact open topology of dimension at most $n^2+2n$, the maximum value occurring only when $D$ is biholomorphically equivalent to the ball and that ${\rm Aut}(D)$ is generically compact and furthermore that the space of such domains that are holomorphically distinct is infinite dimensional; this is true even when the domains are constrained to have circular symmetries (see the survey \cite{GK2}, \cite{IK} and the references therein). On the other hand, the classification of domains $D$ whose group of symmetries is topologically large, i.e., with non-compact automorphism group is more tractable. For instance the well-known theorem of Wong \cite{W} says that if $D$ is a smoothly bounded strongly pseudoconvex domain in $\mathbb{C}^n$ with non-compact automorphism group then $D$ must be equivalent to the unit ball $\mathbb{B}^n$. It is a classical result of H. Cartan that the non-compactness of ${\rm Aut}(D)$ is equivalent to the non-compactness of just one orbit (of some point under the natural action of ${\rm Aut}(D)$ on $D$) when $D$ is bounded, i.e., the existence of a point on $\partial D$ at which an orbit accumulates. Rosay observed \cite{R} shortly after Wong's result that it continues to be true even when $\partial D$ is known to be strongly pseudoconvex only near a boundary orbit accumulation point. Subsequent works on this line have only lent credence to this phenomenon of local data regarding $\partial D$ near a boundary orbit accumulation point providing global information about $D$. A progressive series of analogous extensions (\cite{KK}, \cite{BGK}, \cite{KM}) of the Wong-Rosay characterization of the ball leading to its ultimate version as in the finite dimensional situation (\cite{KM}) has been attained in the setting of a separable Hilbert space as well, where Cartan's theorem need no longer hold. A common technique here has been the scaling method which facilitates the construction of a biholomorphism from a smooth pseudoconvex finite type domain in $\mathbb{C}^n$ with non-compact automorphism group to a non-degenerate rigid polynomial domain -- defined below -- at the boundary orbit accumulation point. For instance, using this and an analysis of holomorphic tangent vector fields, it has been shown (see \cite{BP1}, \cite{BER}) that a bounded domain in $\mathbb{C}^2$, which is smooth weakly pseudoconvex and of finite type $2m$, near a boundary orbit accumulation point must be equivalent to its model domain of the form 
\[
\Big\{ (z_1,z_2) \in \mathbb{C}^2 : 2 \Re z_2 + P(z_1, \ov{z}_1) <0 \Big\}
\]
where $P(z_1, \ov{z}_1)$ is a homogeneous subharmonic polynomial of degree $2m$ without harmonic terms. The pseudoconvexity hypothesis on $\partial D$ near the boundary orbit accumulation point was dropped in \cite{BP2} and more recently in \cite{V}. The Greene-Krantz conjecture, very well-known in this area states that a boundary orbit accumulation point must be of finite type. The classification of non-degenerate rigid polynomial domains being model domains for finite type domains thus gains an added special interest.\\  

\noindent A domain of the form
\[
\Big \{ z \in \mathbb{C}^n \; : \; 2 \Re z_n + P('z, '\ov{z}) < 0 \Big\}
\]
where $P$ is a real valued polynomial in $'z=(z_1, \ldots ,z_{n-1})$ and $'\ov{z}=({}' \ov{z}_1, \ldots, ' \ov{z}_{n-1})$, is called a {\it rigid polynomial domain}. It is called {\it non-degenerate} if there is no germ of a complex analytic variety sitting inside its boundary which is equivalent to the domain being of finite type (see \cite{DA2}). Note that a rigid polynomial domain is simply connected as it admits a deformation retract to the graph of $P$, 
\[
 {\rm Gr}(P)= \Big\{ ('z,x_n) \in \mathbb{R}^{2n-1} \; : \; x_n=-P('z,'\ov{z}) \Big\},
\] 
which is homeomorphic to $\mathbb{C}_{'z} = \mathbb{C}^{n-1}$. \\

We begin by investigating the elements $g$ outside the connected component $G^c$ of the identity in $G={\rm Aut}(\Omega)$ when ${\rm dim}(G)=1$, where 
\[
\Omega = \Big\{ z \in \mathbb{C}^3 \; : \; \rho(z)= 2\Re z_3 + P(z_1,z_2, \ov{z}_1, \ov{z}_2) <0 \Big\}
\]
is a model in $\mathbb{C}^3$, i.e., a Kobayashi hyperbolic, non-degenerate, rigid polynomial domain. Rigidity entails $G^c$ to contain the one-parameter group 
\[
T_t(z_1,z_2,z_3) = (z_1,z_2,z_3 +it)
\]
thereby making $G^c$ non-compact. We shall refer to $T_t$, present in ${\rm Aut}(\Omega)$ for all model domains $\Omega$ as the canonical subgroup. The normality of $G^c$ in $G$ entails that for each $t \in \mathbb{R}$ there exists $t'=f(t)\in \mathbb{R}$ such that $g \circ T_t = T_{t'} \circ g$ -- since ${\rm dim}(G)=1$, $G^c=\{T_t\}$. This expands as
\begin{align}
g_1(z_1,z_2,z_3 + it) &= g_1(z_1,z_2,z_3), \nonumber \\
g_2(z_1,z_2,z_3 + it) &= g_2(z_1,z_2,z_3), \text{ and} \nonumber \\
g_3(z_1,z_2,z_3 + it) &= g_3(z_1,z_2,z_3)+if(t). \label{1.1}
\end{align}
The first two of these equations show that $g_1,g_2$ are independent of $z_3$, so that $'g = (g_1,g_2)$ is a function of $z_1,z_2$ alone. It can be seen from the form of $\Omega$ that it surjects onto $\mathbb{C}^2$ under the natural projection $\pi: \mathbb{C}^3 \to \mathbb{C}_{z_1} \times \mathbb{C}_{z_2}$. This implies that $g_1,g_2$ are entire. Now it follows from the third equation in the system \ref{1.1} that $g_3$ must be of the form $a z_3 + \phi(z_1,z_2)$ so that ${\rm Jac}(g(z))=a {\rm Jac}('g({}'z))$ which is invariant under translations in the $z_3$-direction and if non-empty will intersect $\Omega$. Thus, we conclude that ${\rm Jac}(g)$ is nowhere vanishing, hence constant and so $'g \in {\rm Aut}(\mathbb{C}^2)$. The fact that $\Omega$ is a non-degenerate polynomial domain forces $g$ to be algebraic (see for instance theorem 1.2 in \cite{CP}) whose main tools are a reflection principle and Webster's theorem \cite{We}; so $g_1,g_2$ are entire functions of algebraic growth and consequently $'g \in GA_2(\mathbb{C})$, the group of all polynomial automorphisms of $\mathbb{C}^2$. The classification of elements upto conjugacy in $GA_2(\mathbb{C})$ was done by S. Friedlander and J. Milnor in \cite{FrMi} and this is then used to derive information about the form of $g$ and $P$ -- in case where the components of $'g$ are free from constant terms and $P$ from pluriharmonic terms, it follows that
\[
P \circ {}'g =P.
\] 
Even when ${\rm dim}(G)>1$, all the aforementioned arguments, go through for all $g \in G$ for which (\ref{1.1}) holds, i.e., for all those $g \in G$ that belong to the normalizer in $G$ of the canonical subgroup $\{ T_t \}$; in particular for all $g \in G$ when $G$ is abelian (in which case (\ref{1.1}) holds with $f(t)=t$).\\

\noindent Though our results are valid in a somewhat greater generality, we prefer to focus on the case when $G$ is abelian which is the simplest algebraic condition that can be imposed on $G$. It is known \cite{FM} that if the automorphism group of a hyperbolic domain in $\mathbb{C}^n$ is abelian, then its dimension cannot exceed $n$. Thus ${\rm dim}(G)=1$, $2$ or $3$ and it follows that $G^c$ can be realized as a product of $\mathbb{R}$'s and $\mathbb{S}^1$'s. If $S_s$ is a one-parameter subgroup commuting with $T_t$ then 
\begin{align}
S_s^j(z_1,z_2,z_3+it) &= S_s^j(z_1,z_2,z_3) \; \; \text{for} \; j=1,2, \text{ and} \label{1.2}\\
S_s^3(z_1,z_2,z_3+it) &= S_s^3(z_1,z_2,z_3) + it \nonumber
\end{align} 
for all $s,t \in \mathbb{R}$. (1.2) implies that the first two components of $S_s$ are functions of $z_1,z_2$ alone and as before we can also conclude that $S_s \in GA_2(\mathbb{C})$ for each $s \in \mathbb{R}$, indeed that $S_s$ is a one-parameter subgroup of $GA_2(\mathbb{C})$ which has been a well studied group; in particular a classification upto conjugacy of its one-parameter subgroups is available -- determined by H. Bass and G. Meisters in \cite{BM} -- and we work out the consequences of commutativity of the one-parameter subgroups that are factors of $G$, on both the form of $\Omega$ and $G$ as well. The  normal forms of the commuting subgroups derived in this course, is valid even when $G$ itself is not abelian but for any two commuting one parameter subgroups `different' from $T_t$ that lie in the normalizer of the canonical subgroup $T_t$. The characterization of the model domains below, also remain valid.\\

Let us now state the main results -- all the terminology involved is described in the next section. The first one shows the extent to which even the knowledge of elements $ g \in G \setminus G^c$ is enough to place strong restrictions on $P$ and $G$ as well. For instance, if the group generated by $g$ is a copy of $\mathbb{Z}$ then after a change of variables $P=P(\Im z_1,z_2)$ or $P=P(\vert z_1 \vert^2, z_2)$ and correspondingly, $G^c$ must contain another copy of $\mathbb{R}$ or $\mathbb{R} / \mathbb{Z}$. This provides another consequence of the noncompactness of a group of automorphisms of $\Omega$:  allowing the discrete group $G/G^c$ to contain a copy of (the simplest noncompact discrete group) $\mathbb{Z}$, forces the dimension of $G^c$ to be at least  $2$. All the normal forms of the model domains in the thereoms below are arranged to contain the origin in their boundaries.  
\noindent \begin{thm} \label{dim1}
\item[(a)] Suppose $G$ is abelian and $g \in G$. Then after a change of variables, $'g$ is one of the following:
\begin{itemize}
\item[(i)] A unitary map $U$ with eigenvalues $\alpha_1= e^{ i\theta_1}$, $\alpha_2=e^{ i \theta_2}$. In this case, $P$ is balanced with respect to $(\Theta, \mathbb{Z})$. Let $j=1$ or $2$. If $\alpha_j$ is not a root of unity then $P_j$ must be balanced; if $\alpha_j$ is an $N$-th root of unity then $P_j$ must be balanced in $z_j$ with respect to $\mathbb{Z}_N$.
\item[(ii)] An affine transform $(z_1,z_2) \to (z_1+1, \alpha z_2)$ for some $\alpha \in \mathbb{S}^1$ and $P$ is of the form 
\[
P(z_1,z_2) = P_1(\Im z_1) + M(\Im z_1,z_2) + P_2(z_2)
\]
where $M$ and $P_2$ are balanced in $z_2$ with respect to $\mathbb{Z}_N$, in case $\alpha$ is an $N$-th root of unity and ${\rm dim}(G) \geq 2$; in case $\alpha$ is not a root of unity, they are balanced in $z_2$ and ${\rm dim}(G) \geq 3$. \\

\noindent Furthermore, the third component of $g$ is of the form $z_3 + i \gamma$ for some $\gamma\in \mathbb{R}$ with $\gamma=0$ in case(ii). 
\item[(b)] Suppose ${\rm dim} (G)=1$. Then $G$ must be abelian, only case (i) in (a) can occur and a dichotomy holds: either both $\alpha_1,\alpha_2$ are roots of unity or both of them are not. In the latter case, $P$ is of the form
\[
P_1(\vert z_1 \vert^2) + M(z_1,z_2) + P_2(\vert z_2 \vert^2)
\]
with at least one monomial in $M$ which is not balanced in either variable and with every other mixed monomial being either balanced both in $z_1$ and $z_2$ or neither.
\end{itemize}
\end{thm}
\begin{rem}
Contrast (b) with the case when ${\rm dim}(G)$ is maximal, in which case $\Omega = \mathbb{B}^3$ and $G=PSU(3,1)$ a simple Lie group, so in particular its commutator subgroup is the whole group $G$. By the last statement in (a), $'g$ cannot be the identity map if $g \in G \setminus \{ T_t \}$. In (b), we may also arrange $M$ to be devoid of pluriharmonic terms.\\
It will be seen during the course of the proof that the normalizer of the canonical subgroup $T_t$ coincides with its centralizer and (a) remains valid if we drop the abelianness assumption on $G$ and assume instead that $g$ lies in the normalizer $N$, of $T_t$ in $G$. So, for instance if the group $N/T_t$ contains a copy of $\mathbb{Z}$ then its dimension must be at least $1$ and as soon as $N/T_t$ is non-trivial, we gain information about $P$.
\end{rem}

\begin{thm}\label{dim2}
\begin{itemize}
\item[(a)] Suppose $S_s$ is a one parameter subgroup of $G$ whose infinitesimal generator lies in the normalizer of that of the canonical subgroup $T_t$ and is linearly independent from it. Then $'S_s$ is conjugate to exactly one of the following:
\begin{itemize}
\item[(i)] $(z_1,z_2) \to (z_1,z_2+s)$ in which case 
\[
\Omega \simeq \Big\{ z \in \mathbb{C}^3 \; : \; 2 \Re z_3 + P(z_1, \Im z_2)<0 \Big\}.
\]
\item[(ii)] $(z_1,z_2) \to (z_1, e^{  i \alpha s} z_2)$ where $\alpha\in \mathbb{R}^*$ and in which case 
\[ 
\Omega \simeq  \Big\{ z \in \mathbb{C}^3 \; : \; 2\Re z_3 + P_1(z_1) + M( z_1 , \vert z_2 \vert^2) + P_2(\vert z_2 \vert^2) < 0 \Big\}. 
\] 
\item[(iii)] $(z_1,z_2) \to (z_1+s, e^{ i\alpha s}z_2)$ with $\alpha \in \mathbb{R}^*$ in which case
\[
\Omega \simeq \Big\{ z \in \mathbb{C}^3 \; : \;  2 \Re z_3 + P(\Im z_1, \vert z_2 \vert^2) <0 \Big\}.
\]
\item[(iv)] $(z_1,z_2) \to (e^{ i \alpha s}z_1, e^{  i \beta s}z_2)$ where $ \alpha \beta \in \mathbb{R}^*$ and in this case
\[
\Omega \simeq \Big\{ z \in \mathbb{C}^3 \; : \; 2 \Re z_3 + P_1(\vert z_1 \vert^2) + M(z_1,z_2) + P_2(\vert z_2 \vert^2)<0 \Big\},
\]
where $M \equiv 0$ or $M \not \equiv 0$ and is balanced with respect $\big( (\alpha, \beta), \mathbb{S}^1 \big)$ with $\beta/\alpha \in \mathbb{Q}^*$ , i.e., every $m=cz_1^{j_1}\ov{z}_1^{k_1}z_2^{j_2}\ov{z}_2^{k_2}$ in $M$ satisfies 
\[
(j_1-k_1)\alpha + (j_2-k_2)\beta =0.
\]
So, in particular, every monomial is either balanced both in $z_1$ and $z_2$ or neither. $M$ also be taken to be devoid of pluriharmonic terms.
\end{itemize}
In all the cases the third component of $S_s$ is of the form $z_3+i\beta s$ for some $\beta \in \mathbb{R}$ with $\beta =0$ in cases (i) and (iii).
\item[(b)] Suppose ${\rm dim(G)}=2$ with $G^c$ abelian. Then, in case {\rm (a)(i)}, $G^c \simeq \mathbb{R} \times \mathbb{R}$ while in case {\rm (a)(ii)}, $G^c \simeq \mathbb{R} \times \mathbb{S}^1$ with $P$ not balanced in $z_1$. The case {\rm (a)(iii)} cannot occur and in case {\rm (a)(iv)}, $M \not \equiv 0$ and not extremely balanced and $G^c \simeq \mathbb{R} \times \mathbb{S}^1$. 
\end{itemize}
\end{thm}
Thus when ${\rm dim}(G)=2$, $\Omega$ is equivalent to a model whose $P$ is balanced -- strictly or completely diversely -- in exactly one of the variables or to a model that is strictly balanced with respect to some weight but not extremely balanced in both variables jointly. We also note that when $N/T_t$ contains a copy of $\mathbb{R}$, with its corresponding subgroup $'S_s$ acting non-trivially on $\mathbb{C}^2$, i.e., both components of the vector field 
$F^{-1} \circ{}' S_s \circ F$ is non-zero for all $F \in {\rm Aut}(\mathbb{C}^2)$, then ${\rm dim}(N/T_t)$ is at least $2$. The action is trivial in this sense in (i) and (ii). While, (iv) when considered as an action of $\mathbb{R}$ is non-trivial but not faithful as it reduces to an action of $\mathbb{R}/\mathbb{Z}$ when $\beta/ \alpha \in \mathbb{Q}^*$. In the case $\beta/ \alpha \in \mathbb{R} \setminus \mathbb{Q}$, $M \equiv 0$, the action is non-trivial, faithful and ${\rm dim}(N/T_t)$ is at least $2$ as in case (iii).

\begin{thm} \label{dim3}
Suppose $G^c$ is abelian and three dimensional. Then we have precisely three possibilities:
\begin{itemize}
\item[(i)] $G^c \simeq \mathbb{R} \times \mathbb{R} \times \mathbb{R}$ and 
$\Omega \simeq \Big \{ z \in \mathbb{C}^3 \; : \; 2 \Re z_3 + P(\Im z_1, \Im z_2)<0 \Big \}$
\item[(ii)] $G^c \simeq \mathbb{R} \times \mathbb{R} \times \mathbb{S}^1$ and 
$\Omega \simeq \Big\{ z \in \mathbb{C}^3 \; : \; 2 \Re z_3 + P(\Im z_1, \vert z_2 \vert^2) <0 \Big\}$
\item[(iii)] $G^c \simeq \mathbb{R} \times \mathbb{S}^1 \times \mathbb{S}^1$ and
$\Omega \simeq \Big\{ z \in \mathbb{C}^3 \; : \; 2 \Re z_3 + P(\vert z_1 \vert^2, \vert z_2 \vert^2)<0 \Big \}$.
\end{itemize} 
\end{thm}
Thus for instance in case (i), when $G^c$ is a abelian but torsion free $\Omega$ is extremely balanced but completely diverse. Other cases also illustrate no less, this reflection of the properties of the automorphism group $G$ on the algebraically reduced form of $\Omega$. When the algebraic constraint (abelianness) on $G$ is dropped, the mutual exclusiveness of the above reduced forms, disappears: the ball $\mathbb{B}^n$ can incarnate itself in each of the above three forms and further in the non-extremely balanced form (iv) of theorem \ref{dim2} as well.
\begin{rem} 
The theorem will be seen to be valid when $G^c$ is replaced by the normalizer of $T_t$ as well. Case (i) of this theorem is indeed a simple corollary to the more general characterization of tube domains in theorem 1 of \cite{Y}. We have restricted our attention to model domains in $\mathbb{C}^3$, as the classifications of \cite{FrMi} and \cite{BM} have been used. It would be intersecting to know if analogues of these results hold in higher dimensions as well.
\end{rem}

{\thanks Acknowledgements: I would like to thank my advisor, Kaushal Verma for many helpful discussions. I also acknowledge the support of the Shyama Prasad Mukherjee Fellowship provided by the Council of Scientific and Industrial Research, India.}

\section{Preliminaries}
\noindent We first recall and extend some terminology considered earlier in such a context in \cite{BP4} and \cite{CP}, among others. Let $l=1$ or $2$. We shall say that a monomial $m=cz_1^{j_1}\ov{z}_1^{k_1}z_2^{j_2}\ov{z}_2^{k_2}$ is a mixed monomial if $j_l+k_l>0$ for both values of $l$. Call $m$, {\it pure} if $k_l=0$ for both values of $l$ or $j_l=0$ for both values of $l$; so such a monomial is either holomorphic or anti-holomorphic. We at times assign weights $\theta_l \in \mathbb{R}$ to the variables $z_l$; we will even consider $\theta_l \in \mathbb{C}$ in section $5$ and weights with $\theta_2/\theta_1<0$ will be of importance -- for us these parameters will be provided by the group $G$. Given this assignment, the weight of $m$ is 
\[
{\rm wt}(m)=(j_1 + k_1) \theta_1 + (j_2 + k_2) \theta_2
\]
while its signature is 
\[
{\rm sgn}(m)=(j_1 -k_1) \theta_1 + (j_2 - k_2)\theta_2,
\]
the difference of the weight of the anti-holomorphic component of $m$ from that of its holomorphic component. A real analytic polynomial $p(z_1,z_2)$ is {\it weighted (resp. signature) homogeneous} if all its monomials have the same weight (resp. signature). So $p$ is weighted homogeneous of weight $\lambda$ if it satisfies 
\[
p(e^{\theta_1t}z_1,e^{\theta_2 t}z_2)=e^{\lambda t}p(z_1,z_2)
\]
while it is signature homogeneous of signature $\lambda$ if 
\[
p(e^{i \theta_1 t}z_1,e^{i \theta_2 t}z_2)=e^{ i \lambda t}p(z_1,z_2)
\]
for all $t \in \mathbb{R}$. Now let $A$ stand for one of the groups $\mathbb{R}$ or $\mathbb{Z}$, $\mathbb{S}^1$ or (one of its discrete subgroups $\simeq$) $\mathbb{Z}_N$. Call $m$ {\it balanced} with respect to 
\[
\Big( \Theta = ( \theta_1, \theta_2 ), A \Big)
\]
if its signature is $0$ when $A=\mathbb{R}$ or $\mathbb{S}^1$, an integer when $A=\mathbb{Z}$. A monomial $m$ is said to be balanced with respect to $\mathbb{Z}_N$, more precisely $(\Theta, \mathbb{Z}_N)$, if it is balanced with respect to $(\Theta, \mathbb{Z})$ for some $\Theta=(\theta_1,\theta_2)$ with $e^{i \theta_1},e^{ i \theta_2}$ being a pair of $N$-th roots of unity. Further, $m$ is said to be balanced in $z_1$, if it is balanced with respect to $\big( (1,0),\mathbb{S}^1 \big)$; it is called balanced in $z_1$ with respect to $\mathbb{Z}_N$, if it is balanced with respect to $\big( (\alpha,0),\mathbb{Z}_N \big)$ where $\alpha$ is an $N$-th root of unity (with a similar understanding for being balanced in the variable $z_2$). We shall mention the group $A$ only when it is $\mathbb{Z}$ or $\mathbb{Z}_N$. Note that if $m$ is balanced with respect to both the extremal weights $(1,0)$ and $(0,1)$, i.e., if it is balanced in each of the variables separately, then it is balanced with respect to every weight $(\alpha, \beta)$ and in this case we say that $m$ is extremely balanced. $p$ is said to be {\it strictly balanced} (with respect to a pair $(\Theta, A)$) if each of its constituent monomials is balanced (with respect to that pair), {\it extremely balanced} if all its monomials are so while the notion of a polynomial of {\it balanced diversity} may be introduced as follows. First define the holomorphic quotient $hq(m)$, as the logarithm of the ratio of the weight of the holomorphic component of $m$ to that of its anti-holomorphic component, i.e.,
\[
hq(m)=\log \Big( (j_1\theta_1 + j_2 \theta_2)/(k_1 \theta_1+k_2\theta_2) \Big)
\]
so that $hq(m)=\infty$ (resp. $-\infty$) precisely when $m$ is holomorphic (resp. anti-holomorphic); we do not define $hq(m)$ when $m$ is just a constant and we shall always assume $m$ nonconstant. Note also that $m$ is balanced precisely when $hq(m)=0$. Call $m$ {\it extremely imbalanced} if $hq(m)=\pm \infty$ which happens precisely when $m$ is pure. Call a polynomial $p$ extremely imbalanced if all its monomials are so; example: $p(z_1,z_2) = \Re q(z_1,z_2)$ where $q$ is any holomorphic polynomial. Next, suppose $p$ is a weighted homogeneous polynomial of weight $W$ with respect to some weight $W$; we say that $p$ is {\it completely diversely balanced} if it contains at least one monomial of every possible value of the holomorphic quotient for a monomial of weight $W$, i.e., the set of holomorphic quotients, of all the monomials in it is equal to 
\[
S_W=\{ hq(m) \; : \; m \text{ is a monomial of weight } W \}
\]
which is a symmetric set: $S_W=-S_W$. For example, consider $p(z_1,z_2) = p(\Re z_1, \Re z_2)$. Any of its monomials is of the form $4a_{lm}(\Re z_1)^l(\Re z_2)^m$ which expands to a weighted homogeneous (with respect to any given weight $\Theta$) polynomial $A_{lm}$ in $'z, '\ov{z}$, any of its monomials being upto a constant of the form $z_1^{j_1}z_2^{j_2}\ov{z}_1^{l-j_1}\ov{z}_2^{m-j_2}$ with constant weight $l \theta_1 + m\theta_2=W$, say. Now, if $m=cz_1^{p_1}z_2^{p_2}\ov{z}_1^{q_1}z_2^{q_2}$ is a impure monomial of the same weight $W$, then it is clearly not necessary that $m$ matches (even modulo its coefficient) with one in $A_{lm}$. However, $hq(m)= \log \big( p_1 \theta_1 + p_2 \theta_2/ (l-p_1)\theta_1 + (m-p_2)\theta_2 \big)$ which is also the holomorphic quotient of the monomial $m'=z_1^{p_1}z_2^{p_2}\ov{z}^{l-p_1}\ov{z}_2^{m-p_2}$ which occurs modulo coefficient in $A_{lm}$. Moreover, $A_{lm}$ contains one monomial with $hq=+ \infty$ and one for $hq =-\infty$, so that in all, every $A_{lm}$ and subsequently the $p$ of this example is completely diversely balanced, indeed with respect to any weight and we shall also call such a polynomial {\it extremely diversely balanced}. Finally, call $p$ {\it diversely balanced} with respect to some given weight, if it has no extremely imbalanced monomials (including constants) and the average of the holomorphic quotients of all the monomials in it is zero. With this, every real valued $p$ without pure terms (and constants) will be diversely balanced; however, this can be reconciled with the fact that every real valued real analytic polynomial admits a holomorphic decomposition (introduced by D'Angelo in \cite{DA1} and also discussed in \cite{DA2})
\[
p(z,\ov{z}) = 2 \Re q(z) + \vert p_1(z) \vert^2 - \vert p_2(z) \vert^2 
\]
for some uniquely determined holomorphic polynomial $q$ and some holomorphic maps $p_1,p_2$ with $p_1(0)=0=p_2(0)$ (this ensures that $2 \Re q(z)$ does not decompose as $\vert f \vert^2 - \vert g \vert^2$ for some holomorphic maps, $f$ and $g$), thus rendering a decomposition of $p$ into (diversely) balanced and (extremely) imbalanced parts.\\

We call a model domain $\Omega$ as above, completely diversely/strictly/extremely balanced model if the corresponding $P$ is so; the non-degeneracy assumption of $\Omega$ rules out the possibility of $P$ being totally imbalanced. Non-degeneracy of $\Omega$ forces all the level sets $L$ of $P$ in $\mathbb{C}^2$ to be free from any germ of a non-trivial complex analytic variety. We shall drop without mention, the emphasis that $P$ is real analytic and write $P('z)$ for $P('z, '\ov{z})$ and often split $P$ as a sum of three parts
\[
P(z_1,z_2) = P_1(z_1) + M(z_1,z_2) + P_2(z_2)
\]
where $P_l$ is the sum of all those monomials involving $z_l, \ov{z}_l$ alone  and $M$ consists of all the (remaining) mixed monomials in $P$. Another simple consequence of the finite type assumption that will be used often is that $P_l(z_l) \not\equiv 0$ for $l=1,2$; this means that terms involving $z_l,\bar{z}_l$ alone occur in $P$. Also, we assume $P(0)=0$ so that the origin lies in $\partial \Omega$. Indeed, our change of variables $C$ in our reductions of the form of a member or a one-parameter subgroup of $G$ may well reintroduce a constant term in $P$ but their $'C$ will be independent of $z_3$ and $C_3 =z_3$, so for every such change of variables, we may by the change of the $z_3$-variable given by the translation
\[
(z_1,z_2,z_3) \to (z_1,z_2,z_3 + P(0))
\]
ensure that $P$ has no constant terms. The first two components of the automorphism will at all stages of the reduction procedure depend only on $z_1,z_2$, so $C$, clearly does not disturb the reduced form of the first two components of the automorphism; in fact even the possibility of a real constant now getting added to the third component of the reduced form will be seen to be ruled out owing to it being decoupled from the $z_3$ variable in all cases, so that conjugation by the above translation leaves it intact. Similar considerations apply to ensure the passage of an initial assumption about $P$ being free from pluriharmonic terms through all stages of calculations and change of variables involved in reducing the form of the automorphisms so that finally we have both their reduced form and this assumption about $P$ holding. Indeed, the sum of such terms will be of the form $2 \Re \phi(z_1,z_2)$ and we keep making the change of variables $C_\phi$, obtained by replacing $P(0)$ in the aforementioned translation by $\phi(z_1,z_2)$, without disturbing the reduced form of the automorphisms (whose first two components will be independent of $z_3$ and the third component of the form $z_3 + \psi(z_1,z_2)$ at all stages of their reduction process, so that conjugation by $C_\phi$ does not affect the form of the automorphism). The removal of pluriharmonic terms in $P$ aids in a direct transfer of the symmetries of the domain to those of $P$ especially when they are rotational while it is desirable to retain them when dealing with translational symmetries.\\ 

\noindent Finally, let us note again here for clarity that the fact that $'g \in GA_2(\mathbb{C})$ for all $g \in {\rm Aut}(\Omega)$, does not require the pseudoconvexity of $\partial \Omega$. Indeed, recall theorem 1.2 of \cite{CP} that every proper holomorphic mapping between any two non-degenerate rigid polynomial domains is algebraic ($g$ extends across $\omega^c$, the pseudoconcave portion of $\partial \Omega$ which is an open subset thereof and maps $\omega^c$ into itself. Non-degeneracy of $\partial \Omega$ ensures the same for its Levi form on an open dense subset and puts us in the situation of Webster's theorem \cite{W}; it is argued in \cite{CP} that it is possible pass to the situation of Webster's theorem even when $\omega^c =\phi$). Now as noted earlier, $'g \in {\rm Aut}(\mathbb{C}^2)$ so that the components $g_j$ ($j=1,2$) of $'g$ are entire functions satisfying equations of the from
\[
a^j_0('z) \big( g_j('z) \big)^{k_j} + a^j_1('z) \big( g_j('z) \big)^{k_j-1} + \ldots + a^j_{k_j}('z) =0
\]
where the $a^j_l$'s are holomorphic polynomials. Recalling the elementary estimate on the location of the roots $\zeta$ of a holomorphic polynomial $z^k + a_1 z^{k-1} + \ldots + a_k$ that $ \vert \zeta \vert  \leq 2 \; {\rm max}_j \vert a_j \vert^{1/j}$ we have that $g_j('z)$ is an entire function of algebraic growth. More precisely,  
\[
\vert a^j_0('z)\vert \vert g_j('z) \vert \leq 2 \; {\underset {l} {\rm max }} \big\vert a^j_l('z) \big\vert
\]
for all $'z$ outside the zero variety of $a^j_0('z)$ which is a thin set. It follows that $a^j_0g_j$ must be a polynomial. Since $g_j$ is entire we now have that $g_j$ is itself a polynomial. \\  

The group of all polynomial automorphisms of $\mathbb{C}^n$ will be denoted by $GA_n(\mathbb{C})$. All change of variables will be through polynomial automorphisms. $Q,R,a_j,b_j$ etc. will stand for polynomials whose definitions will keep varying (but remain fixed between successive definitions). All sums occurring below are finite.

\section{Automorphisms Not Connected to the Identity -- Proof of theorem \ref{dim1}}
The first step in the proof of theorem \ref{dim1} will consist of translating the fact that $g$ preserves $\Omega$, to an simple as equation as possible. This is equation (\ref{2.4}) which is easily obtained when $G$ is abelian (in which case equation (\ref{2.3}) coincides with (\ref{2.4})). To obtain the same when ${\rm dim}(G)=1$ -- indeed, that $G$ must be abelian -- and also to unravel further, the constraints imposed on $P$ and $g$ by that equation to reach certain definite conclusions about their form, we break the proof into two cases depending on whether $'g$ is conjugate to an elementary polynomial automorphism of $\mathbb{C}^2$ or not. We conclude by ruling out the latter possibility.\\
We shall not assume $G$ abelian and only work with $g \in N$, the normalizer of $T_t$. The case when $G$ is abelian will be simple corollary to this discussion. We begin with a differentiated form of (\ref{1.1}) for $g$,  
\begin{equation}\label{2.1}
\partial g_3/ \partial z_j (z_1,z_2,z_3+it) = \partial g_3/ \partial z_j (z_1,z_2,z_3) 
\end{equation}
which shows that $\partial g_3/ \partial z_j$ is independent of $z_3$ for $j=1,2$ and so $g_3$ is of the form
\[
g_3(z_1,z_2,z_3) = \phi(z_1,z_2) + \psi(z_3)
\]
for some holomorphic maps $\phi$ and $\psi$. Using the same constraint on the $z_3$-derivative at (\ref{2.1}), we see that $\psi$ is linear in $z_3$ (absorbing if necessary, the constant in $\psi$ into $\phi$), i.e.,
\[
g_3( z_1,z_2,z_3 ) = \phi( z_1,z_2 ) + az_3
\]
Feeding this back into (\ref{1.1}) we get $f(t)=at$. In the case when $G$ is abelian, $f(t)=t$ and so $a=1$.
Next note that for all $t \in \mathbb{R}$,
\[
\big( {}'z, -P('z, '\ov{z})/2 +it \big) \in \partial \Omega .
\]
Therefore, since $g$ preserves $\partial \Omega$,
\begin{equation}\label{2.2}
2 \Re \big( \phi('z) + a(-P('z,'\ov{z})/2 +it) \big) + P \big( {}'g('z), {} '\ov{g('z)} \big) =0.
\end{equation}
Let $a=\mu +i \nu$. Then,
\[
2 \Re \phi('z) + 2 \big( -\mu P('z,{}'\ov{z})/2 - \nu t \big) + P \big({}'g('z), {}'\ov{g('z)} \big) =0.
\]
Comparing coefficients of $t$ on both sides, we get $\nu=0$, so $a=\mu \in \mathbb{R}$ and (\ref{2.2}) becomes
\begin{equation} \label{2.3}
2 \Re \phi('z) = \mu P('z,'\ov{z}) - P({}'g('z), '\ov{g('z)} ). 
\end{equation}
Consider
\[
F('z, 'w) = \phi('z) + \ov{\phi('\ov{w})} - \mu P('z,'w) + P \big({}'g('z), {}' \ov{g(' \ov{w})} \big)
\]
which is holomorphic in $('z,'w) \in \mathbb{C}^2 \times \mathbb{C}^2$ and vanishes on $\{'\ov{w}= {}'z\}$ which is maximally totally real and so vanishes identically on $\mathbb{C}^2 \times \mathbb{C}^2$. Putting $'w=0$ and noting that $P('z,0) \equiv 0$ as $P$ has no pure terms, we have
\[
\phi('z) = -\ov{\phi(0)} - P \big( {}'g('z), \ov{\omega} \big)
\]
where $\omega = {}' g (0) $, so $\phi$ is a polynomial and hence $g \in GA_3(\mathbb{C})$.\\
Now, suppose one of the components of $'g$, say $g_2$, is a function of one of the variables alone, say $g_2(z_1,z_2)=g_2(z_2)$ and moreover has a fixed point $z_2^0$. Then consider the domain
\[
\Omega_0 = \big\{ (z_1,z_3) \in \mathbb{C}^2 : 2 \Re z_3 +P_0(z_1, \ov{z}_1) <0 \big\}
\]
where $P_0(z_1, \ov{z}_1)=P(z_1,z_2^0, \ov{z}_1,\ov{z}_2^0)$. It can be seen that this is also a finite type domain. Since $g \in G$ we have
\[
2 \Re \big(g_3(z_1,z_2,z_3) \big) + P \big( g_1(z_1,z_2), g_2(z_1,z_2) \big) <0
\]
for all $(z_1,z_2,z_3) \in \mathbb{C}^3$ with 
\[
2 \Re z_3 + P(z_1,z_2) <0.
\]
Since 
\[
P \big(g_1(z_1,z_2^0),g_2(z_1,z_2^0) \big) = P \big( g_1(z_1,z_2^0), z_2^0 \big) = P_0 \big( g_1(z_1,z_2^0), \overline{g_1(z_1,z_2^0)} \big)
\]
we have for points of the form $(z_1,z_2^0,z_3) \in \Omega$
\[
2 \Re(g_3(z_1,z_2^0,z_3)) + P_0 \big( g_1(z_1,z_2^0), \overline{g_1(z_1,z_2^0)} \big) <0,
\]
i.e., $g^0(z_1,z_3) = \big( g_1(z_1,z_2^0), g_3(z_1,z_2^0,z_3) \big)$ is an automorphism of $\Omega_0$. Keeping this observation aside, let us now consider two cases depending on the conjugacy class in $GA_2(\mathbb{C})$ to which $g$ belongs.

\noindent \textbf{Case (A):} First we deal with the case when $'g$ is conjugate to an elementary map, i.e., after a change of variables, $'g$ is given by 
\[
'g(z_1,z_2)=(\gamma z_1 + \delta, q(z_1) + \tau z_2)
\]
where $\gamma,\tau \in \mathbb{C}^*, \delta  \in \mathbb{C}$ and $q(z_1) \in \mathbb{C}[z_1]$. Before passing, we note that an affine map can be conjugated to an elementary-affine map by conjugating its linear component	to its Jordan normal form.
Now Friedlander and Milnor have shown (see \cite{FrMi}) that we can by a further change of variables if necessary, reduce $'g$ further to one of the following forms
\begin{itemize}
\item[(a)]
\begin{itemize}
\item[(i)] A diagonal linear map $(z_1,z_2) \to (\alpha z_1, \beta z_2)$ with $\alpha \beta \in \mathbb{C}^*$,
\item[(ii)] An aperiodic affine transform $(z_1,z_2) \to (z_1+1, \alpha z_2)$ with $\alpha \in \mathbb{C}^*$,
\end{itemize} 
\item[(b)] $(z_1,z_2) \to \big( \beta^d(z_1+z_2^d), \beta z_2 \big)$ with $d \in \mathbb{N}$, $\beta \in \mathbb{C}^*$
\item[(c)] $(z_1,z_2) \to \big( \beta^\nu (z_1 + z_2^\nu q(z_2^r)), \beta z_2 \big)$ for $\nu \geq 0$ and in this case $\beta$ is a primitive $r$-th root of unity and $q(z)$ is a non-constant polynomial of the form
\[
z^k+ q_{k-1}z^{k-1}+ \ldots + q_1z +1
\]
with $q_{k-1}=0$ when $\beta=r=1$.
\end{itemize}
Now, note that in each of these cases, at least one of the components (indeed, the second component) is a function of one of the variables only, so by the foregoing observation we have that the automorphism $g^0$ of $\Omega^0$ as above, is of the form
\[
g^0(z_1,z_3) = (\gamma z_1 +\delta, \phi(z_1,z_2^0) + \mu z_3 \big).
\]
For an automorphism $g^0$ of the polynomial domain $\Omega_0 \subset \mathbb{C}^2$ of this form, we have by the proof of proposition 2.7 of \cite{V} that $ \vert \gamma \vert =1 = \vert \mu \vert$. The fact that 
\[
g= ( {}'g('z), \phi('z) + \mu z_3)
\]
preserves $\Omega$ gives for all $z \in \Omega$ that 
\[
2 \Re(\mu z_3) + 2 \Re\phi('z) + P \big({}'g('z), ' \ov{g('z)} \big) <0
\]
which simplifies by (\ref{2.3}) to
\[
\mu \big( 2 \Re z_3 + P('z, ' \ov{z}) \big) <0
\]
for $z \in \Omega$. So $\mu$ must be positive and hence $\mu=1$. So $f(t)=t$ and subsequently $g$ commutes with $T_t$. Thus, with the hindsight that $'g$ must necessarily be conjugate to an elementary map, we see that the normalizer of the canonical subgroup coincides with its centralizer; as the normalizer of $T_t$ is $G^c$ when ${\rm dim}(G)=1$, $G$ must be abelian in this case.\\
Next, (\ref{2.3}) now reads
\begin{equation} \label{2.4}
P('z, '\ov{z}) - P({}'g('z),{}'\ov{g('z)}) = 2 \Re \phi('z).
\end{equation}
Note that a constant term in $P$ if any cancels out on the left. When $P$ has no pluriharmonic terms, the same is true of $P \circ {}'g $ as well provided $'g(0)=0$, in which case we have by the above equation that $2 \Re \phi('z)=0$, i.e., $\phi$ is an imaginary constant and subsequently,
\[
P \circ {}' g=P
\]
We shall presently work out the consequences of this or (\ref{2.4}) on the form of $P$. Before that let us record a simple fact that will used many times. 
\begin{lem}\label{useful}
Suppose $Q(z_1,z_2)$ is a real analytic polynomial such that for some $p(z_1) \in \mathbb{C}[z_1,\ov{z}_1]$ we have 
\begin{equation}\label{psym}
Q(z_1,p(z_1)z_2)= \sum\limits_j a_j(z_1, \ov{z}_1) (\Im z_2)^j
\end{equation}
Then each of the polynomials $a_j$ is divisible by $\vert p(z_1) \vert^{2j}$ i.e.,
\[
Q(z_1,z_2) = \sum\limits_j b_j(z_1, \ov{z}_1) \big( \Im z_2 \ov{p(z_1)} \big)^j
\]
for some real analytic polynomials $b_j$.
\end{lem}
\begin{proof}
For $p(z_1) \neq 0$, rewrite (\ref{psym}) as follows
\begin{align*}
Q(z_1,z_2) &= \sum\limits_j a_j(z_1,\ov{z}_1) \big( \Im(z_2/p(z_1) \big)^j \\
&= \sum\limits_j (a_j(z_1, \ov{z}_1)/\vert p(z_1) \vert^{2j}) \big( \Im z_2 \overline{p(z_1)} \big)^j \\
&= \sum\limits_j ( a_j(z_1, \ov{z}_1)/\vert p(z_1) \vert^{2j} ) \big( \Im z_2 \Re p(z_1) - \Re z_2 \Im p(z_1) \big)^j \\
&= \sum\limits_{j} \sum\limits_{k} \Big({}^j C_k (-1)^{j-k} a_j(z_1, \ov{z}_1)/ \vert p(z_1) \vert^{2j} (\Re p(z_1))^k (\Im p(z_1) )^{j-k} \Big) (\Im z_2)^k (\Re z_2)^{j-k}
\end{align*}
Now, the coefficient of $( \Im z_2)^k (\Re z_2)^{j-k}$ in $Q(z_1,z_2)$ is 
\[
{}^j C_k (-1)^{j-k} a_j(z_1, \ov{z}_1) / \vert p(z_1) \vert^{2j} (\Re p(z_1))^k (\Im p(z_1))^{j-k}.
\]
Next, notice that $\vert p(z_1) \vert^2$ has no common factors with $\Re p(z_1)$ or $\Im p(z_1)$ and hence $\vert p(z_1) \vert^{2j}$ must divide $a_j(z_1, \ov{z}_1)$ in $\mathbb{C} [z_1, \ov{z}_1]$. Indeed, pick any prime factor of $p$, which must be of the form $z_1 - \alpha$ where $\alpha$ is one of the zeros of $p$ -- nothing is lost by assuming $p$ to be monic. Expand $p$ about $\alpha$, i.e., $p(z_1)= \sum c_j(z_1 - \alpha)^j$. Now suppose $\Re p(z_1)$ factors as $(z_1 - \alpha)q(z_1, \ov{z}_1)$ with $q(z_1, \ov{z}_1) \in \mathbb{C}[z_1,\ov{z}_1]$. Make the linear change of variables
\[
w=z_1 -\alpha
\]
to obtain 
\[
\sum\limits_j c_j w^j + \sum\limits_j \ov{c}_j \ov{w}^j = wq(w+\alpha, \ov{w+\alpha})
\]
Now, every monomial on the right is divisible by $w$. Noting that this cannot be the case with the left hand side, finishes the verification that neither $p$ nor $\ov{p}$ can share a common factor with its real or its imaginary part and thereby the lemma follows. 
\end{proof}
\begin{rem}
Thus we note that the two basic examples of a real valued extremely imbalanced and a non-extremely diversely balanced polynomial of one variable namely $\Re p(z)$ and $\vert q(z) \vert^2$ for holomorphic polynomials $p,q$ are `independent' in the sense that their greatest common divisor is $1$ upto a unit in $\mathbb{C}[z]$.
\end{rem} 
\noindent Continuing with the proof of theorem (\ref{dim1}), we first argue that $'g$ cannot be conjugate to the maps in (b), (c) and (a)(ii) when ${\rm dim}(G)=1$.\\

{\it The case when $'g$ is conjugate to a map of the form (a)(ii)}: 
After a change of variables if necessary,  
\[
'g^n(z_1,z_2) = (z_1 + n, \alpha^n z_2)
\]
where $'g^n$ denotes the $n$-fold composition of $g$ with itself. Note that since the first component, being a translation has no fixed point, we cannot conclude via the aforementioned arguments that $\vert \alpha \vert=1$. However, (\ref{2.4}) applied to $'g^n ={}'(g^n)$ gives
\begin{equation}\label{Ptransym}
P(z_1,z_2) - P(z_1 + n, \alpha^n z_2) = 2 \Re \phi(z_1 + n-1, \alpha^{n-1}z_2) 
\end{equation}
for all $n \in \mathbb{Z}$. In the present case we will not assume that $P$ is devoid of pluriharmonic terms; in fact we want their presence to enable us to cast $P$ in a form that will make its symmetries apparent. Letting $D^2$ denote any one of the operators $\partial ^2/ \partial z_j \partial \ov{z}_k$ with $(j,k) \in \{1,2\} \times \{1,2\}$ we have
\begin{equation} \label{2.45}
c_\alpha (D^2P)(z_1+n, \alpha^{n}z_2) = (D^2P)(z_1,z_2)
\end{equation}
where $c_\alpha$ is a non-zero constant, the precise value of which is as specified in the table below: 
\begin{table}[ht] \label{table}
\caption{Values of $c_\alpha$}
\centering
\begin{tabular}{c c}
\hline \hline
$(j,k)$ & $c_\alpha$ \\
\hline \\
(1,1) & $1$\\
(2,2) & $\vert \alpha \vert^{2n}$ \\
(1,2) & $\ov{\alpha}^n$ \\
(2,1) & $\alpha^n$\\
\end{tabular}
\end{table}\\
Let
\[
Q(z_1,z_2)=D^2P(z_1,z_2) = \sum\limits_j a_j(z_2, \ov{z}_2, \Re z_1)(\Im z_1)^j.
\]
By (\ref{2.45}),  
\[
c_ \alpha \sum\limits_j a_j(\alpha^{n}z_2,\ov{\alpha}^{n}\ov{z}_2, \Re z_1+n)(\Im z_1)^j = \sum\limits_j a_j(z_2,\ov{z}_2, \Re z_1)(\Im z_1)^j
\]
Equating coefficients of $(\Im z_1)^j$ we have 
\[
c_\alpha a_j(\alpha^{n}z_2,\ov{\alpha}^{n} \ov{z}_2, \Re z_1 + n) = a_j(z_2,\ov{z}_2, \Re z_1)
\]
for all $n \in \mathbb{Z}$. Expand $a_j$ as $\sum b_{kl}^j(\Re z_1) z_2^k \ov{z}_2^l$. We have by comparing coefficients of $z_2^k \ov{z}_2^l$ that
\begin{equation} \label{2.5}
c_\alpha \alpha^{nk} \ov{\alpha}^{nl} b_{kl}^j(x_1 + n) = b_{kl}^j(x_1)
\end{equation}
where $x_1=\Re z_1$ and we know $\alpha \neq 0$. Therefore, $x_1^0 + n$ is a (complex) root of $b_{kl}^j$ whenever $x_1^0$ is a root of the polynomial $b_{kl}^j$ which implies that the $b_{kl}^j$'s are all constants and so the $a_j$'s are independent of $\Re z_1$. Therefore the polynomial $Q$ must be of the form
\begin{equation} \label{2ndderivP}
Q(z_1,z_2) = \sum\limits_j a_j(z_2, \ov{z}_2)(\Im z_1)^j.
\end{equation}
At least one of the constants $b_{kl}^j$'s has got to be non-zero for otherwise all the $a_j$'s and subsequently $\partial^2 P/\partial z_j \partial \ov{z}_k$ for all values of $(j,k)$, will have to be zero, implying that $P$ is pluriharmonic which contradicts the finite type assumption. Now, (\ref{2.5}) gives $\vert \alpha \vert=1$.\\
Let's see what this implies for $P$. Consider first the case that
\[
Q= \partial^2 P/ \partial z_1 \partial \ov{z}_1
\]
Note that the anti-derivative of $(\Im z_1)^j$ with respect to $z_1$ or $\ov{z}_1$ is again a monomial in $\Im z_1$. A term-by-term integration with respect to $z_1,\ov{z}_1$ of the above form of $Q$ therefore leads to the expression of $P$ as (recall that $P$ has no constant term)
\[
P(z_1,z_2)=\sum\limits_j a^{11}_j(z_2)(\Im z_1)^j + C_1(z_1,z_2)
\]
for some real analytic polynomials $a^{11}_j$ and $C_1$, with every monomial in $C_1$ being pure in $z_1$. Putting $z_2=0$, we get that $P(z_1,0)$, which is constituted by precisely all those monomials in $P$ that are independent of $z_2$, is of the form
\[
P_1(z_1)= \sum\limits_j a^{11}_j(0)(\Im z_1)^j + q_1(z_1) + \ov{q_2(z_1)}
\]
for some holomorphic polynomials $q_1,q_2$. Since $P_1(z_1)=P(z_1,0)$ is real valued so is 
\[
P_1(z_1 +t) - P_1(z_1)=q_1(z_1)-q_1(z_1+ t ) +\ov{q_2(z_1)} -\ov{q_2(z_1 + t)}
\] 
for all $t \in \mathbb{R}$ which gives 
\[
(q_1-q_2)(z_1)=(q_1-q_2)(z_1+t)
\]
for all $t \in \mathbb{R}$ showing that  
\[
P_1(z_1)=\sum\limits_j a^{11}_j(0)(\Im z_1)^j + 2 \Re q_1(z_1).
\]
After the change of variables 
\[
(z_1,z_2,z_3) \to \big( z_1,z_2,z_3 + q_1(z_1) \big),
\]
we will have $P_1(z_1) = Q_1(\Im z_1)$ for some real valued, real analytic polynomial $Q_1$. Such a change of variables does not affect $'g$ and therefore allows us another application of (\ref{2ndderivP}). 

Now consider $Q=\partial^2 P/\partial z_2 \partial \ov{z}_2$. The relevant $c_\alpha$ at (\ref{2.5}) is $\vert \alpha \vert^{2n}=1$, so that equation  reads $\alpha^{n(k-l)}=1$ which means that
\[
Q=\partial^2 P/\partial z_2 \partial \ov{z}_2 = \sum\limits_j a_j(z_2, \ov{z}_2) (\Im z_1)^j
\]
with the monomials $cz_2^k \ov{z}_2^l$ in $a_j$ satisfying the condition that $k-l$ is divisible by $m$ if $\alpha$ is an $m$-th root of unity, else $k=l$. In either case, $a_j(\alpha^n z_2, \ov{\alpha}^{n} \ov{z}_2)=a_j(z_2,\ov{z}_2)$. Integrating this as before leads to $P$ being expressed as 
\[
P(z_1,z_2) = \sum\limits_j a_j^{22}(z_2)(\Im z_1)^j + C_2(z_1,z_2)
\]
for some real analytic polynomials $a_j^{22}$ and $C_2$ with every monomial in $C_2$ being pure in $z_2$. Put $z_1=0$ and denote the sum of the pure $z_2$ monomials coming from the above equation namely, $C_2(0,z_2)$ by $q'_1(z_2) + \ov{q'_2(z_2)}$ for holomorphic polynomials $q'_1,q'_2$. We then have by the argument used earlier in this connection that 
\[
(q'_1-q'_2)(\alpha^n z_2)=(q'_1-q'_2)(z_2)
\]
for all $n \in \mathbb{Z}$ which says that $q'_1(z_2)=q'_2(z_2) + \tilde{q}(z_2)$ where $\tilde{q}(z_2)=\sum c_j z_2^j$ with the sum running over indices $j$ that are divisible by $m$ if $\alpha$ is a root of unity, else $\tilde{q} =0$. Thus 
\[
P_2(z_2)=P(0,z_2)=a^{22}_0(z_2, \ov{z}_2) + 2 \Re q'_2(z_2) + \tilde{q}(z_2).
\]
Removing the middle term by the change of variables 
\[
(z_1,z_2,z_3) \to \big( z_1,z_2,z_3 +q'_2(z_2) \big)
\]
and observing that the monomials in $a_j^{22}$ and $a_j$ differ just by a (balanced) factor of $c \vert z_2 \vert^2$ -- so that they share the same properties -- we get in particular that $P_2(\alpha^n z_2)=P_2(z_2)$. The standing state of the equation (\ref{Ptransym}) namely,
\[
M(z_1,z_2) + P_2(z_2) - M(z_1+n, \alpha^n z_2) - P_2(\alpha^n z_2) = 2 \Re \phi(z_1 + {n-1},\alpha^n z_2)
\]
now reduces to one that is free of $P_2$ paving the way to get a better hold on the mixed terms
\begin{equation} \label{Mtransym}
M(z_1,z_2) - M(z_1+n, \alpha^n z_2) = 2 \Re \phi(z_1+n-1, \alpha^{n-1} z_2).
\end{equation}
Now take $Q=\partial^2 P/ \partial z_1 \partial \ov{z}_2 = \partial^2 M/ \partial z_1 \partial \ov{z}_2$ and integrate (\ref{2ndderivP}) once with respect to $\ov{z}_2$ and then $z_1$, to get that $M$ must be of the form
\[
M(z_1,z_2) = \sum\limits_j a^{12}_j(z_2,\ov{z}_2)(\Im z_1)^j + Q_{21} 
\]
where $Q_{21}$ is constituted by those terms in $M$ which are annihilated by $\partial^2/ \partial z_1 \partial \ov{z}_2$. Now, by the version of (\ref{2.5}) for our present $Q$, in which $c_\alpha=\ov{\alpha}^n$ we get, remembering that $b_{kl}^j$'s were constants, that 
\[
\alpha^{nk} \ov{\alpha}^{n(l+1)}=1.
\]
Writing $\alpha=e^{ i \theta}$, we have for every $(k,l)$ for which $b_{kl}^j \neq 0$ that
\[
(e^{ i \theta})^{n(k-l-1)} =1
\]
which shows that $\alpha$ must be a root of unity -- say a primitive $m$-th root -- unless we always have $k=l+1$. So, $b_{kl}^j$ can be non-zero only if $(k-l-1)$ is a multiple of $m$. Since the monomials in $a_j$ (of our present $Q$) and $a_j^{12}$ differ by a factor of $c\ov{z}_2$, the monomials $cz_2^k\ov{z}_2^l$ have the property that $k-l$ is divisible by $m$ in case $\alpha$ is an $m$-th root of unity, else $k=l$; in either case $a_j^{12}(\alpha^n z_2)=a_j^{12}(z_2)$ yielding a further reduced version of (\ref{Mtransym}), namely
\begin{equation}\label{Q12transym}
Q_{21}(z_1,z_2) - Q_{21}(z_1+n, \alpha^n z_2) =2 \Re \phi(z_1+n-1,\alpha^{n-1}z_2)
\end{equation}
We boot-strap this a final time by applying its consequence (\ref{2ndderivP}) with $Q=\partial^2 Q_{21}/\partial \ov{z}_1 \partial z_2$ again. The corresponding polynomials $a_j$ now satisfy 
\[
a_j(z_2,\ov{z}_2) = \alpha^n a_j(\alpha^n z_2,\ov{\alpha}^n \ov{z}_2)
\]
Also observe that if we expand $Q_{21}$ in a similar manner (see \ref{Q12form} below) then the $a_j$'s differ from the $a_j^{21}$ occurring below at (\ref{Q12form}) by a factor of $c z_2$. Therefore, the monomials $cz_2^k\ov{z}_2^l$ in $a_j$ will share the same properties as those of $a_j^{12}$ and $a_j^{21}$. The form of $Q_{21}$ obtained by integration of (\ref{2ndderivP}) as
\begin{equation} \label{Q12form}
Q_{21}(z_1,z_2) = \sum\limits_j a_j^{21}(z_2,\ov{z}_2) (\Im z_1)^j + \text{ terms annhilated by } \partial^2/\partial \ov{z}_1 \partial z_2.
\end{equation}
can this time be quickly reduced to $Q_{12}(z_1,z_2)=\sum a_j^{21}(z_2,\ov{z}_2) (\Im z_1)^j$ by merely recalling that all the terms in $Q_{21}$ were divisible by $\ov{z}_1z_2$ to begin with. Feeding this back into (\ref{Q12transym}) then yields at once that $\phi \equiv 0$. Therefore 
\[
P(z_1,z_2)=P_1(\Im z_1) + \sum\limits_{j \geq 1} a_j(z_2)(\Im z_1)^j +  b(z_2) 
\]
for some real analytic polynomials $a_j$ and $b$ all of whose monomials $cz_2^k\ov{z}_2^l$ share the property that $k-l$ is divisible by $m$ in the case when $\alpha$ is an $m$-th root of unity, else $k=l$. Evidently, we have a one parameter group of translations in the $\Re z_1$-direction namely
\[
(z_1,z_2,z_3) \to (z_1+t,z_2,z_3),
\]
apart from the group $\{T_t\}$. So, if ${\rm dim}(G)=1$ then $g \in G \setminus G^c$ cannot be such that $'g$ is conjugate to a map of the form (a)(ii).\\

{\it $'g$ is not conjugate to a map of the form (b)}: Supposing the contrary, we write down its $n$-th iterate 
\[
'g^n(z_1,z_2) = \big( \beta^{nd}(z_1+nz_2^d), \beta^nz_2 \big)
\]
Note that in this case, $'g^n(0)=0$ so $\mu=1$ and the third component of $'g$ is $z_3 + \phi(z_1,z_2)$ and $P \circ {}'g^n =P$, i.e., 
\begin{equation}\label{2.7}
P \big( \beta^{nd}(z_1 + nz_2^d), \beta^nz_2 \big) = P(z_1,z_2) 
\end{equation}
for all $n \in \mathbb{Z}$. Replacing $z_1$ by $z_1-nz_2^d$ we rewrite this as 
\begin{equation*}
P \big( \beta^{nd}z_1, \beta^nz_2 \big)=P(z_1-nz_2^d, z_2)
\end{equation*}
and then replacing $z_1$ by $z_2^dz_1$ we have
\begin{equation}\label{2.8}
P \big( \beta^{nd}z_2^dz_1, \beta^nz_2 \big) = P \big( z_2^d(z_1-n),z_2 \big)  
\end{equation}
for all $n \in \mathbb{Z}$ and all $(z_1,z_2) \in \mathbb{C}^2$. By the finite type assumption the right hand side does not reduce to constant, by putting $z_2 = \zeta$ for any $\zeta \in \mathbb{C}^*$. This enables us to pick any non-zero complex number $\zeta$ and $z_1^0=x_1^0 + iy_1^0$ such that the polynomial in one real variable $x$ defined by
\[
R(x) =P(\zeta^d(x+iy_1^0), \zeta)
\]
is non-constant. Indeed, if such a $z_1^0$ does not exist, then it means that $P(\zeta^d z_1, \zeta)$ is independent of $\Re z_1$ for all $\zeta \in \mathbb{C}^*$. So $P(\zeta^dz, \zeta) = \sum a_j(\zeta, \ov{\zeta})(\Im z)^j$ and by lemma \ref{useful} we have for some real analytic polynomials $b_j$ that
\[
P(z_1, z_2) = \sum b_j(z_2, \ov{z}_2) (\Im \ov{z}_2^d z_1)^j.
\]
Since $P$ is real valued, $P(0,0)=b_0(0,0)= \alpha \in \mathbb{R}$ and the variety parametrised by $ t \to (t, 0, - \alpha/2)$ for $t \in \mathbb{C}$, lies inside $\partial \Omega$ contradicting its finite type condition. Thus a choice of $z_1^0$ as above to make the polynomial $R(x)$ non-constant is indeed possible. Now notice by (\ref{2.8}) that 
\[
R(x_1^0 - n) =P \big( \zeta^d(x_1^0 - n +i y_1^0), \zeta \big)=P(\zeta^d(z_1^0 - n), \zeta)=P((\beta^n \zeta)^dz_1^0, \beta^n \zeta).
\]
Now, by comparing the highest degree terms in $P$ involving $z_1, \ov{z}_1$ alone at (\ref{2.7}), we have $\vert \beta \vert =1$; so the sequence $(\beta^{nd} \zeta^d z_1^0, \beta^n \zeta)$ is bounded and hence admits a convergent subsequence to $(w_1,w_2) \in \mathbb{C}^2$, say. Correspondingly, a subsequence of $R(x_1^0 - n)$ would then converge to $P(w_1,w_2)$ contradicting the fact that 
\[
\displaystyle\lim_{n \to \infty} \vert R(x_1^0 - n)\vert = \infty.
\]
Thus we conclude that there cannot exist $g \in G \setminus G^c$ with its $'g$ conjugate to a map of the form (b).\\

{\it $'g$ is not conjugate to a map of the form (c)}: If not, after a change of variables  
\[
'g(z_1,z_2)=\big( \beta^\nu (z_1 + z_2^\nu q(z_2^r)), \beta z_2 \big)
\]
where $\beta$ is a primitive $r$-th root of unity and the $n$-th iterate of $'g$ is 
\[
'g^n(z_1,z_2)=\big( \beta^{n \nu}(z_1 + n z_2^\nu q(z_2^r)), \beta^n z_2 \big).  
\]
If $\nu \neq 0$ then $'g(0)=0$ and we have for all $n \in \mathbb{Z}$ that 
\[
P\big(\beta^{n \nu}(z_1 + n z_2^\nu q(z_2^r)), \beta^n z_2 \big) =P(z_1, z_2),
\]
which is equivalent to
\[
P \big( \beta^{ n \nu}z_1, \beta^n z_2 \big) = P \big( z_1 - n z_2^\nu q(z_2^r),z_2 \big).
\]
Replacing $z_1$ by $z_2^\nu q(z_2^r)z_1$, we have
\[
P \big( \beta^{n \nu}z_2^\nu q(z_2^r)z_1, \beta^nz_2 \big) = P \big( z_2^\nu q(z_2^r)(z_1- n),z_2 \big).
\]
The finite type assumption will ensure via an argument as in the previous case, a choice of $\zeta$ and $z_1^0=x_1^0 + iy_1^0$ such that $\tilde{\zeta}=\zeta^\nu q(\zeta^r) \neq 0$ and such that the polynomial $R(x)$ this time defined as
\[
R(x) = P(\tilde{\zeta}(x+iy_1^0), \zeta)
\]
is non-constant. Then $R(x_1^0 - n)= P(\tilde{\zeta}(z_1^0 -n), \zeta)$ and 
\[
P(\beta^{n \nu} \tilde{\zeta} z_1^0, \beta^n \zeta) = R(x_1^0 - n)
\]
and then we have a contradiction as in the previous case.\\
When $\nu=0$ so that $'g^n (z_1,z_2) = \big( z_1 + nq(z_2^r), \beta^n z_2 \big)$, we need no longer have $'g(0)=0$. We then only have that $P$ must satisfy
\[
P(z_1,z_2) - P(z_1 + nq(z_2^r), \beta^n z_2) = 2 \Re \phi \big( z_1 + (n-1)q(z_2^r), \beta^{n-1} z_2 \big)
\]
for all $n \in \mathbb{Z}$ and some holomorphic polynomial $\phi$. Replacing $z_1$ by $r q(z_2^r)z_1$ we have
\[
P(rq(z_2^r)(z_1 +n),z_2) - P(rq(z_2^r)z_1,z_2) = 2 \Re \phi \big( rq(z_2^r)z_1 + (n-1)rq(z_2^r), z_2 \big)
\]
Let $Q(z_1,z_2) = P(rq(z_2^r)z_1,z_2)$ and $\psi(z_1,z_2) =  \phi \big( rq(z_2^r)(z_1 -1),z_2 \big)$ we get
\[
Q(z_1+n,z_2)- Q(z_1,z_2) = 2 \Re \psi(z_1+n,z_2)
\]
Write
\[
Q(z_1,z_2) = \sum a_j(z_2, \ov{z}_2, \Re z_1)(\Im z_1)^j
\]
and 
\[
\Re\psi(z_1,z_2)= \sum b_j(z_2, \ov{z}_2, \Re z_1).
\]
Then
\[
\sum \big( a_j(z_2, \ov{z}_2, \Re z_1+n) - a_j(z_2, \ov{z}_2) \big) (\Im z_1)^j = \sum b_j(z_2, \ov{z}_2, \Re z_1 +n)(\Im z_1)^j
\]
Equating coefficients of $(\Im z_1)^j$ and writing 
\[
p_j=a_j-b_j = \sum\limits_{k,l} c_{k,l}(x_1) z_2^k \ov{z}_2^l
\]
and 
\[
a_j(z_2, \ov{z}_2,x_1)=\sum\limits_{k,l} d_{kl}(x_1) z_2^k \ov{z}_2^l
\]
we get $c_{jk}(x_1 + n)= d_{jk}(x_1)$ for all $n \in \mathbb{Z}$. Now, if $x_1^0$ is a root of $d_{jk}$ then $x_1^0+n$ is a root of $c_{jk}$ for all $n \in \mathbb{Z}$ showing that $d_{jk}$'s are constants and therefore
\[
P(rq(z_2^r)z_1,z_2)= \sum a_j(z_2, \ov{z}_2)(\Im z_1)^j.
\]
Lemma \ref{useful} then gives 
\[
P(z_1,z_2)= \sum a_j(z_2, \ov{z}_2) \big( \Im z_1 \ov{rq(z_2^r)} \big)^j.
\]
The complex line $t \to (t, z_2^0, -\alpha/2)$, where $z_2^0$ is a root of the non-constant polynomial $q(z_2^r)$, lies inside $\partial \Omega$ contradicting its finite type character.\\ 

{\it The case when $g \in G\setminus G^c$ is such that $'g$ is of the form (a)(i)}: 
\[
'g(z_1,z_2) = (\alpha z_1, \beta z_2)
\]
Since $'g(0)=0$ in this case, $P \circ {}'g^n=P$ for all $n \in \mathbb{Z}$, i.e.,
\[
P(\alpha^n z_1, \beta^n z_2) =P(z_1, z_2) 
\]
Consider a monomial $cz_1^{j_1} \ov{z}_1^{k_1} z_2^{j_2} \ov{z}_2^{k_2}$ occurring in $P$. The transformation 
\[
'g^n(z_1,z_2)=(\alpha^n z_1, \beta^n z_2)
\]
changes the coefficient of this monomial on the left hand side by a factor of $\alpha^{nj_1}\ov{\alpha}^{nk_1} \beta^{nj_2} \ov{\beta}^{nk_2}$. Comparing coefficients we then have
\[
\alpha^{nj_1} \ov{\alpha}^{nk_1} \beta^{nj_2} \ov{\beta}^{nk_2} =1.
\]
Now, $\Omega$ being of finite type, $P$ contains monomials involving each of $z_1,z_2$ alone, so we have $\ov{\alpha}=1/ \alpha$, $\ov{\beta}=1/\beta$ and therefore
\[
\alpha^{n(j_1-k_1)} \beta^{n(j_2-k_2)} =1  
\]
for all $n \in \mathbb{Z}$. Writing $\alpha = e^{ i \theta_1}$, $\beta=e^{ i \theta_2}$ for some $\theta_1,\theta_2 \in \mathbb{R}$, this reads
\[
e^{i[(j_1-k_1)\theta_1 + (j_2-k_2)\theta_2]} =1
\]
which implies that $(j_1-k_1) \theta_1 +(j_2-k_2) \theta_2 $ is an integer. Considering monomials in $z_1,\ov{z}_1$ alone i.e., those for which $j_2=k_2=0$ which as we know do occur in $P$, we have $\alpha^{j_1-k_1}=1$ giving rise to 2 cases: either $\alpha$ is a root of unity or else $j_1=k_1$. Write $P$ as
\[
P(z_1,z_2) =P_1(z_1,\ov{z}_1) + M(z_1, \ov{z}_1) +P_2(z_2, \ov{z}_2)
\]
where $P_1$ is the sum of all those monomials not involving $z_2,\ov{z}_2$, $P_2(z_2, \ov{z}_2)$ the sum of those that involve only $z_2, \ov{z}_2$ and $M(z_1,z_2)$ the remaining (mixed terms). Then 
\[
P_1(z_1,\ov{z}_1)=P_1(\vert z_1 \vert^2)
\]
if $\alpha$ is not a root of unity. Similarly $P_2$ will also have to be balanced, if $\beta$ is not a root of unity. In case when $\alpha$ is an $N$-th root of unity every monomial $cz_1^{j}\ov{z}_1^{k}$ in $P_1$ will have the property that $j-k$ is divisible by $N$. A similar argument holds if $\beta$ is a root of unity.\\

\noindent Now consider the case when ${\rm dim} (G)=1$. If $P$ were balanced in $z_1$, i.e., $P(z_1, \ov{z}_1) = P( \vert z_1 \vert^2)$ and $M(z_1,z_2)=M(\vert z_1 \vert^2, z_2)$ then the $1$-parameter group
\[
(z_1,z_2,z_3) \to (e^{i \theta}z_1,z_2,z_3)
\]
increases the dimension of $G$ by $1$. In particular if $\alpha$ is not a root of unity, $M$ -- which has now got to be non-zero -- cannot be balanced in $z_1$ and so has a monomial $cz_1^{j_1} \ov{z}_1^{k_1} z_2^{j_2} \ov{z}_2^{k_2}$ with $j_1 \neq k_1$ and its presence gives rise to the equation $\alpha^{j_1 - k_1} \beta ^{j_2 -k_2}=1$. Since $\alpha$ is not a root of unity $j_2 \neq k_2$. Thus we have a monomial that is neither balanced in $z_1$ nor in $z_2$ and that $\beta$ is also not a root of unity. We can now see that every mixed monomial is either balanced both in $z_1$ and in $z_2$ or neither.\\

\noindent \textbf{Case(B):} We finish by ruling out the other possibility that $'g$ is conjugate to a composition $H$ of generalized H${\rm \acute{e}}$non maps (see theorem 2.6 of \cite{FrMi}). Express $H$ as a reduced word $H=h_1 \circ \ldots \circ h_n$ where each $h_i$ belongs to the affine subgroup $A$ or to the subgroup $E$ consisting of all elementary automorphisms (but not to $S=A \cap E$), $n \geq2$ and no two of the consecutive factors belong to the same subgroup $A$ or $E$. Since we may cyclically permute the factors of the reduced word without changing its conjugacy class, we may assume that the word is cyclically reduced, i.e., the extreme factors $h_1$ and $h_n$ belong to different subgroups (among $A$ and $E$). It is now clear that a reduced word representation for $H^m$ can be obtained by juxtaposing that of $H$, $m$-many times. The degree of a polynomial automorphism is by definition the maximum of the degrees of the component polynomials and by theorem 2.1 of \cite{FrMi},
\[
{\rm deg}(H)={\rm deg}(h_1)\ldots {\rm deg}(h_n).
\]
So ${\rm deg}(H)>1$ and ${\rm deg}(H^m)=({\rm deg}H)^m$. Now (\ref{2.4}) applied to $H^m$ for any $m \in \mathbb{Z}$, reads  
\[
P\circ H^m = \mu P('z, '\ov{z}) - 2 \Re \phi('z).
\]
Suppose that the degree of the right side is $d$. Then, as $P$ has terms involving each of $z_1$ alone and $z_2$ alone, $P \circ H^m$ will have terms of degree at least $({\rm deg}H)^m$ which will be bigger than $d$ for all large $m$, giving a degree mismatch contradiction. 
\qed

\section{Models when ${\rm dim}(G)=2$ -- Proof of theorem \ref{dim2}}
\noindent Suppose first that $\rm{dim}(G)\geq 2$ ($G$ not necessarily abelian) and contains a one parameter subgroup $\{S_s\}$ that lies in the normalizer of the canonical subgroup $\{ T_t\}$ and `different' from it, meaning that their infinitesimal generators are linearly independent. Then, as observed in the previous section, the normalizer of $T_t$ is same as its centralizer and so $S_s$ commutes with $T_t$, i.e.,
\begin{align}
S_s^j(z_1,z_2,z_3 +it) &= S_s^j(z_1,z_2,z_3)  \;\; \text{ for } j=1,2  \nonumber \\
S_s^3(z_1,z_2,z_3+it) &= S_s^3(z_1,z_2) +it.
\end{align}
The first equation shows that the first two components of $S_s$ are independent of $z_3$, so $(S_s^1(z_1,z_2),S_s^2(z_1,z_2)) \in GA_2(\mathbb{C})$, while the one for the last component shows that the flow is decoupled from the $z_3$-direction, i.e.,
\begin{equation}\label{3.0}
S_s^3(z_1,z_2,z_3) = z_3 + h(s,z_1,z_2)
\end{equation}
Now since $S_s$ preserves $\partial \Omega$, we have
\[
2 \Re \big( z_3 + h_s(z_1,z_2) \big) + P \big( S_s^1(z_1,z_2),S_s^2(z_1,z_2) \big) =0
\]
whenever $2 \Re z_3 + P(z_1,z_2)=0$. So we may rewrite this as
\begin{equation} \label{meqn}
P(z_1,z_2) - P \big( S_s^1(z_1,z_2), S_s^2(z_1,z_2) \big) = 2 \Re h_s(z_1,z_2) 
\end{equation}
which shows that $h_s$ must be a polynomial in $(z_1,z_2)$ for all $s \in \mathbb{R}$. To simplify the form of $'S_s$, recall the following classification of 1-parameter subgroups of $GA_2(\mathbb{C})$ from \cite{BM}.
\begin{thm} \label{nrmlfrm}
After a change of variables, every 1-parameter subgroup of $GA_2(\mathbb{C})$ falls into one of the following categories:
\begin{itemize}
\item[(1)] $(z_1,z_2) \to (z_1,e^{bt}z_2)$ where $b \in \mathbb{C}^*$
\item[(2)] 
\begin{itemize}
\item[(a)] $(z_1,z_2) \to (z_1+t, e^{bt}z_2)$ with $b \in \mathbb{C}^*$
\item[(b)] $(z_1,z_2) \to (z_1+t,z_2)$
\end{itemize}
\item[(3)] $(z_1,z_2) \to (z_1,z_2 + p(z_1)t)$ where $p$ is a monic polynomial of degree $\geq 1$.
\item[(4)] $(z_1,z_2) \to (e^{at}z_1,e^{bt}z_2)$ where $a,b \in \mathbb{C}^*$
\item[(5)] $(z_1,z_2) \to \big(e^{at}z_1, e^{adt}(z_2 + tz_1^d) \big)$ where $a \in \mathbb{C}^*$ and $d \in \mathbb{N}$.
\end{itemize}
\end{thm}
\noindent After a change of variables assume that our subgroup $S_s$ is such that $'S_s$ is in one of these forms and consider the cases when $'S_s(0,0)=(0,0)$. These are (1), (3) when $p(0)=0$, (4) and (5). In these cases 
\begin{equation} \label{meqn2}
P \circ {}'S_s =P
\end{equation}
holds since $P \circ {}'S_s$ will have no pluriharmonic terms, just as $P$ has none and so the right hand side of (\ref{meqn}) vanishes identically and $h_s \equiv i \beta s$ for some $\beta \in \mathbb{R}$. We shall now work out the consequences of (\ref{meqn}) and (\ref{meqn2}) on the form of $P$. Many arguments have been outlined in the previous section but there is still some room for reasonable refinement, particularly in the case of rotational symmetries and simpler arguments.

\noindent \textit{Case (i)}: Let us start by ruling out the possibility of a subgroup $S_s$ of the form (5) being contained in $G$, in which case we have for all $s \in \mathbb{R}$ and $a \in \mathbb{C}$ that 
\[
P \big( e^{as}z_1,e^{ads}(z_2 + sz_1^d) \big)=P(z_1,z_2).
\]
Write 
\[
P(z_1,z_2) = \sum a_{jklm} z_1^j\ov{z}_1^kz_2^l\ov{z}_2^m.
\]
and put $z_2=0$. Then the last equation becomes 
\[
\sum (a_{jklm}s^{l+m}) e^{\{(j+ld)a + (k+md)\ov{a}\}s} z_1^{j+ld} \ov{z}_1^{k+md} = \sum a_{jk00}z_1^j\ov{z}_1^k.
\]
Comparing the coefficient of $z_1^p \ov{z}_1^q$ on both sides now gives an equation of the form
\[
\Big( \sum_{ \substack{j+ld=p \\
                       k+md=q }} a_{jklm}s^{l+m} \Big) e^{(pa + q\ov{a})s} =a_{pq00}. 
\]
The left hand side is of the form $r(s)e^{\beta s}$ and is a constant function only if $\beta=0$ and $r(s) \equiv \rm{constant}$; $\beta=0$ means $pa+q \ov{a}=0$, i.e., $(p+q)\Re a + i (p-q) \Im a =0$ giving $\Re a=0$ and $p=q$. So,
\[
r(s)= \sum\limits_{j+ld=k+md=p} a_{jklm}s^{l+m}
\]
which gives $a_{jklm}=0$ whenever $l+m>0$. But the finite type constraint shows that we must have at least one $a_{jklm} \neq 0$ with $l+m>0$, i.e., with $l>0$ or $m>0$ since terms involving $z_2, \ov{z}_2$ alone must occur. Thus $\{ 'S_s \}$ cannot be (after our normalizing change of variables) of the form (5).\\

\noindent \textit{Case (ii)}: Next we tackle the cases when $'S_s$ is of the form (1) or (4). Then
\begin{equation}\label{prsym}
P( e^{as}z_1,e^{bs}z_2 )=P(z_1,z_2)
\end{equation}
for all $s \in \mathbb{R}$ with $a,b \in \mathbb{C}$ and at least one of them non-zero, say $b$. Write 
\[
P(z_1,z_2)=P_1(z_1) + M(z_1,z_2) +P_2(z_2)
\]
as before. Then putting $z_1=0$ in (\ref{prsym}), we have for all $s \in \mathbb{R}$ that 
\[
P_2(e^{bs}z_2, \overline{e^{bs} z_2})=P_2(z_2, \ov{z}_2).
\]
Comparison of the coefficient of every monomial of the kind $cz_2^j\ov{z}_2^k$ occurring in $P_2$ in the above equation gives rise to equations of the form 
\[
(e^{bs})^j(e^{\ov{b}s})^k =1
\]
for all $s \in \mathbb{R}$. Writing $b=x+iy$ we have 
\[
e^{(j+k)xs} e^{i(j-k)ys}=1
\]
for all $s \in \mathbb{R}$. Since $j+k>0$, $x=0$. So $y \neq 0$ and
\[
e^{i(j-k)ys}=1
\]
for all $s \in \mathbb{R}$. Therefore $j=k$, so every monomial that occurs in $P_2$ is balanced i.e., $P_2(z_2, \ov{z}_2)=P_2(\vert z_2 \vert^2)$.
A similar consideration of terms in $P_1$ shows that $a$ is also an imaginary constant and if non-zero, $P_1(z_1,\ov{z}_1)=P_1(\vert z_1 \vert^2)$. Write $a=  i \alpha $, $b= i \beta $ for some $\alpha, \beta \in \mathbb{R}$. Considering next, a monomial of the kind $cz_1^j\ov{z}_1^k z_2^l \ov{z}_2^m$ occurring in $M(z_1,z_2)$ we have
\[
e^{ i(j-k) \alpha s}e^{ i(l-m) \beta s} =1
\]
for all $s \in \mathbb{R}$ which gives
\[
(j-k) \alpha + (l-m) \beta =0.
\]
Suppose $\alpha,\beta$ are both non-zero. Then, $j=k$ implies $l=m$ and conversely, i.e., each term in $M$ is either balanced both in $z_1$ and $z_2$ or balanced neither in $z_1$ nor in $z_2$.
Now suppose $S_s$ is conjugate to a subgroup of the form (1) and assume $\beta =0$ (so $\alpha \neq 0$). Then the foregoing equation is $(j-k) \alpha =0$, i.e., $j=k$. So $P_1=P_1(\vert z_1 \vert^2)$ and the mixed terms must be balanced in $z_1$.\\

\textit{Conclusion:} If $'S_s$ is conjugate to a subgroup of the form (1) then after a change of variables,
\[
P(z_1,z_2) = P(z_1) + M(z_1, \vert z_2 \vert^2) + P_2(\vert z_2 \vert^2)
\]
while for the case when $'S_s$ is conjugate to a subgroup of the form (4), i.e., $\alpha,\beta$ both are non-zero, then $P_1=P_1(\vert z_1 \vert^2)$ and $P_2=P_2(\vert z_2 \vert^2)$. If $M \equiv 0$ or $M=M(\vert z_1 \vert^2, \vert z_2 \vert^2)$ so that 
\[
P(z_1,z_2) = P(\vert z_1 \vert^2) + M(\vert z_1 \vert^2,\vert z_2 \vert^2) +P_2(\vert z_2 \vert^2).
\]
then as is evident, ${\rm dim}(G)$ becomes at least $3$. So when ${\rm dim}(G)=2$, $M$ must be non-zero and must not be balanced. Consequently, we have an equation of the form 
\[
(j-k)\alpha + (l-m) \beta =0
\]
with $j \neq k$ (and so $l \neq m$ as well). Thus all the monomials $cz_1^{j_1} \ov{z}_1^{k_1} z_2^{j_2} \ov{z}_2^{k_2}$ occurring in $P$ are balanced with respect to the weights $(\alpha, \beta)$ for $(z_1,z_2)$. Notice that in this case $\beta/\al$ is rational. 
Needless to say, that $M$ is neither balanced in $z_1$ nor in $z_2$ and $G^c \equiv \mathbb{R} \times \mathbb{S}^1$.\\

\noindent \textit{Case(iii)}: Consider now the case when $'S_s$ is conjugate to a subgroup of the form (2). In this case we no longer have $'S_s(0,0) =(0,0)$ and therefore equation (\ref{meqn2}) no longer holds. We work with (\ref{meqn}). But before that let us examine what happens to the third component of $S_s$ -- whose form we had only pinned down when equation (\ref{meqn2}) was known to hold. To this end, check what $S_t \circ S_s = S_{t+s}$ implies for the third component. We have
\[
S_{t+s}^3(z_1,z_2,z_3) = S_t^3 \big( S_s^1(z),S_s^2(z),S_s^3(z) \big)
\]
By (\ref{3.0}),
\begin{align*}
z_3 + h(t+s,z_1,z_2) &= S_s^3 + h(t,S_s^1,S_s^2)\\
&= z_3 + h(s,z_1,z_2) + h(t,S_s^1,S_s^2).
\end{align*}
Recalling that $h(0,z_1,z_2) \equiv 0$, we have
\begin{align*}
h(t+s,z_1,z_2) -h(s,z_1,z_2) &= h(t,z_1+s,e^{bs}z_2)\\
\big( h(t+s,z_1,z_2) - h(s,z_1,z_2) \big)/t &= \big( h(t,z_1+s,e^{bs}z_2) - h(0,z_1+s,e^{bs}z_2) \big)/t
\end{align*}
and hence
\[
\partial h/ \partial t(s,z_1,z_2) = \partial h / \partial t (0,z_1+s,e^{bs}z_2).
\]
Since the right hand side is a polynomial and the anti-derivatives of $(z_1+s)^n$ and $ (e^{bs}z_2)^n$ as functions of $s$ are polynomials of $(z_1 +s)$ and $e^{bs}z_2$ respectively, we have by integrating with respect to $s$ the above equation that 
\[
h(s,z_1,z_2) = q(z_1+s,e^{bs}z_2) + C(z_1,z_2)
\]
for some $q(z_1,z_2) \in \mathbb{C}[z_1,z_2]$. Put $s=0$. Then $h_0 \equiv 0$ gives $C(z_1,z_2) = -q(z_1,z_2)$ and so 
\[
h(s,z_1,z_2) =q(z_1+s,e^{bs}z_2) - q(z_1,z_2).
\]
Hence $S_s(z_1,z_2,z_3) = \big( z_1+s,e^{bs}z_2,z_3+ q(z_1+s,e^{bs}z_2)-q(z_1,z_2) \big)$.
Now, the automorphism $ (z_1,z_2,z_3)  \to (\tilde{z}_1, \tilde{z}_2, \tilde{z}_3) = \big( z_1,z_2,z_3 -q(z_1,z_2) \big)$ conjugates $S_s$ to the automorphism
\[
(\tilde{z}_1,\tilde{z}_2) \to (\tilde{z}_1+s, e^{bs} \tilde{z}_2, \tilde{z}_3)
\]
so that if 
\[
 \tilde{\Om} = \{ z \in \mathbb{C}^3 \; : \; 2 \Re z_3 + \tilde{Q}(\tilde{z}_1, \tilde{z}_2)<0 \}
\] 
then $\tilde{Q}(\tilde{z}_1, \tilde{z}_2) = \tilde{Q}( \tilde{z}_1+s, e^{bs} \tilde{z}_2)$. Of course, $\tilde{Q}$ may have pluriharmonic terms unlike $P$ but $\tilde{\Omega}$ is of finite type. Dropping the $\tilde{}$'s and writing $Q(z_1,z_2) = \sum a_{jk}(z_1, \ov{z}_1)z_2^j \ov{z}_2^k$ for $\tilde{Q}$ we have by the last equation that
\[
\sum a_{jk} (z_1, \ov{z}_1) z_2^j \ov{z}_2^k = \sum a_{jk} (z_1 +s, \ov{z_1 + s} )e^{(bj+\ov{b}k)s} z_2^j \ov{z}_2^k.
\]
Comparing coefficients, we then have
\[
a_{jk}(z_1,\ov{z}_1) = a_{jk}(z_1+s,\overline{z_1 +s}) e^{(bj +\ov{b}k)s}
\]
for all $(j,k)$. This gives $bj +\ov{b}k =0$.\\
\textit{Sub-case (a)}: $b \neq0 $. Write $b =x+iy $ to get $(x+iy)j = - (x-iy)k$. Then $(j+k)x =0$ and $(j-k)y=0$ which gives $x=0$ and $j=k$. That is, $b$ is an imaginary constant and $a_{jk} \equiv 0$ whenever $j \neq k$ and when $j=k$ we have that the polynomials $a_j=a_{jj}$ satisfy 
\[
a_{j}(z_1, \ov{z}_1) = a_{j}(z_1+s, \overline{z_1+s})
\]
Expanding both sides in powers in $\Im z_1$ coefficients that are polynomials in $\Re z_1$, we readily get by equating coefficients that they are constants,  i.e., the $a_j$'s are free of $\Re z_1$ and hence $Q$ is of the form
\[
Q(z_1,z_2) = \sum a_j(\Im z_1) \vert z_2 \vert^{2j}.
\]
But then we clearly have three one-parameter subgroups in $G$ whose generating vector fields are linearly independent. Thus, the present case can happen only when ${\rm dim}(G) \geq 3$.\\
\textit{Sub-case(b)}: $b=0$, i.e., the subgroup $S_s$ is conjugate to the group 
\[
(z_1,z_2) \to (z_1 + s, z_2).
\]
In this case we can no longer assert that $a_{jk}=0$ for $j \neq k$ and we have
\[
a_{jk}(z_1,\ov{z}_1) = a_{jk}(z_1+s, \overline{z_1+s})
\]
which as before implies that $a_{jk}=a_{jk}(\Im z_1)$ so that $Q$ is of the form
\[
Q(z_1,z_2) = \sum a_{jk}(\Im z_1) z_2^j \ov{z}_2^k
\]
or put differently $Q(z_1,z_2) = \sum a_j(z_2, \ov{z}_2) (\Im z_1)^j$ and $G^c \simeq \mathbb{R} \times \mathbb{R}$.\\

\noindent \textit{Case (iv)}: Finally, we consider the case when there is a one-parameter subgroup of the form (3). In this case again, we do not necessarily have $S_s(0,0) = (0,0)$ as $p(0)$ may be non-zero and we follow the same procedure as in the previous case. First we examine the third component of $S_s^3$ which we know to be of the form $z_3 + h_s(z_1,z_2)$ with $h_s \in \mathbb{C} [z_1,z_2]$. Using the fact that $ S_s $ is a one-parameter group, we have as before for its third component that
\[
\partial h/ \partial t (s,z_1,z_2) =\partial h/ \partial t(0,z_1,z_2 +p(z_1)s).
\]
Since the right hand side is a polynomial and the anti-derivative of $(z_2 + p(z_1)s)^n$ considered as a function of $s$ is again a polynomial function of $(z_2 + p(z_1)s)$, integration of the last equation gives for some $q(z_1,z_2) \in \mathbb{C}[z_1,z_2]$ that
\[
h(s,z_1,z_2) = q(z_1,z_2 + p(z_1)s) + C(z_1,z_2)
\]
Putting $s=0$ and remembering $h(0,z_1,z_2) \equiv 0$ we get $C(z_1,z_2) = -q(z_1,z_2)$ and so
\[
h(s,z_1,z_2) = q(z_1,z_2 + p(z_1)s) - q(z_1,z_2)
\]
which enables to pass as before to an equivalent domain where (\ref{meqn2}) holds. Alternately we may apply (\ref{meqn}), i.e.,
\begin{equation} \label{3.3}
2 \Re \big( q(z_1, z_2+p(z_1)s) - q(z_1,z_2) \big) = P(z_1,z_2) - P(z_1,z_2 + p(z_1)s)
\end{equation}
Let 
\[
Q(z_1,z_2) = P(z_1,z_2) + 2 \Re(q(z_1,z_2)) 
\]
and note that terms involving $z_2, \ov{z}_2$ alone do occur in $Q$. For firstly they occur in $P$ and then because these are all non-pluriharmonic, they cannot be cancelled by any term in the polynomial $2 \Re q(z_1,z_2)$ which will contain only pluriharmonic terms. Next rewrite (\ref{3.3}) as
\[
Q(z_1,z_2 + p(z_1)s) = Q(z_1,z_2).
\]
So if we let $\tilde{Q}(z_1,z_2) = Q(z_1,p(z_1)z_2)$ this can be further rewritten as 
\[
\tilde{Q}(z_1,z_2 +s) = \tilde{Q}(z_1,z_2)
\]
Expand $\tilde{Q}$ as $\tilde{Q}(z_1,z_2) = \sum a_j(z_1, \Re z_2) (\Im z_2)^j$ to get
\[
\sum a_j(z_1, \Re(z_2 + s)) (\Im z_1)^j = \sum a_j(z_1, \Re z_2) (\Im z_2)^j
\]
Comparing coefficients of $(\Im z_1)^j$ on both sides we have
\[
a_j(z_1,\Re z_2 +s) =a_j(z_1, \Re z_2)
\]
for all $s \in \mathbb{R}$ which shows that $a_j$'s are independent of $\Re z_2$ and
\begin{equation}\label{3.4}
\tilde{Q}(z_1,z_2) = \sum a_j(z_1, \ov{z}_1) (\Im z_2)^j
\end{equation}
Lemma (\ref{useful}) then allows us to conclude that $Q$ must be of the form
\[
Q(z_1,z_2) = Q(z_1,p(z_1)z_2) = \sum b_j(z_1, \ov{z}_1) (\Im z_2 \overline{p(z_1)})^j
\]
for some real analytic polynomials $b_j$. Now, $Q$ is real valued, so in particular $Q(z_1,0)=b_0(z_1, \ov{z}_1)$ is also real valued. Since $p$ is a non-constant polynomial, it has at least one root say $z_1^0$. Consider the variety parametrised by $\psi( \zeta) =(z_1^0, \zeta, -\alpha/2)$ where $\alpha =b_0(z_1^0,\ov{z}_1^0) \in \mathbb{R}$ and $\zeta \in \mathbb{C}$; then
\[
\rho( \psi(\zeta)) = 2(-\alpha/2) + \sum b_j(z_1^0, \ov{z}_1^0) \big( \Im  \zeta \ov{p(z_1^0)} \big)^j =0
\]
i.e., the variety parametrised by $\psi$ lies inside $\partial \Omega$, contradicting its finite type condition. Therefore, there cannot be a subgroup in $G^c$ that can be conjugated to a subgroup of the form (3).

\section{Commuting Flows in the plane}
\noindent As a prelude to the next section we work out normal forms for commuting pairs of one parameter subgroups in $GA_2(\mathbb{C})$. A one parameter subgroup will be said to be of type $(j)$, $1 \leq j \leq 5$, if it can be conjugated to the subgroup of case $(j)$ listed in theorem (\ref{nrmlfrm}) while we say that it is in the form $(j)$ when no such change of variables is necessary i.e., it is already in that form; by form (2) we mean one of the forms (2)(a) or (2)(b). For the sake of completeness we determine the form of such subgroups that commute with a type (5) or a type (3) subgroup as well. To start with, we assume a change of variables already made so that one of the subgroups $S_s$, is in its normal form -- one out of the five given by theorem \ref{nrmlfrm} -- and we discuss the possible forms/types of the other one parameter subgroup $R_t$ that commutes with $S_s$. In all the cases, we will see that a single change of variables will suffice to put both the subgroups in their normal form.\\

\noindent \textit{Case(i)}:  Suppose $S_s(z_1,z_2) = (z_1,e^{b s} z_2)$. Then
\begin{align}
R_t^1 (z_1,z_2) &= R_t^1(z_1,e^{b s} z_2) \label{creln11}, \text{ and}\\
e^{b s} R_t^2(z_1,z_2) &= R_t(z_1,e^{b s} z_2) \label{creln12}.
\end{align}
for all $s,t \in \mathbb{R}$. Writing $R_t^1(z_1,z_2)= \sum a_{jk}(t) z_1^jz_2^k$, (\ref{creln11}) gives
\[
\sum a_{jk}(t)z_1^jz_2^k = \sum e^{k b s} a_{jk}(t)z_1^jz_2^k.
\]
Equating coefficients gives $a_{jk}(t)(e^{kbs}-1)=0$ for all $s,t \in \mathbb{R}$. Since $e^{k b s}=1$  only when $k=0$, i.e., $a_{jk} \equiv 0$ for all $k \geq 1$. Therefore, 
\[
R_t^1(z_1,z_2)= \sum a_j(t)z_1^j =p_t(z_1),
\]
say.\\
Similarly writing $R_t(z_1,z_2)= \sum b_{jk}(t)z_1^jz_2^k$ we have by (\ref{creln12}) that  $b_{jk}(t) \equiv 0$ for all $k \neq 1$ and $R_t^2(z_1,z_2)= \sum b_j(t)z_1^jz_2=z_2q_t(z_1)$, say. Now since $R_t(z_1,z_2) \in {\rm Aut}(\mathbb{C}^2)$ its Jacobian which is $q_t(z_1) \partial p_t(z_1)/\partial z_1$ must be a function of $t$ alone, i.e., independent of $z_1$. So $q_t(z_1)=q_0(t)$ and $p_t(z_1)=a(t)z_1 + b(t)$ where $a(t)$ and $q_0(t)$ are nowhere vanishing and $p_0(z_1)=z_1$, $b(0)=0$ and $a(0)=1$. Next, since $p_t$ is a one parameter subgroup we have
\[
a(s)\big( a(t)z_1 + b(t) \big) +b(s) = a(t+s)z_1 + b(t+s)
\]
which shows that $a(t)$ is a one parameter subgroup of $\mathbb{C}^*$ and so is of the form $e^{\lambda t}$ for some $\lambda \in \mathbb{C}$ and $a(s)b(t) + b(s) =b(t+s)$ which may now be recast as
\[
\big( b(t+s)-b(s) \big)/s= e^{\lambda s} \big( b(t)-b(0) \big)/s.
\]
This gives $b'(s)=ce^{\lambda s}$ where $c=b'(0)$ and therefore
\begin{align*}
b(s) &= (c/\lambda) (e^{\lambda s} -1) \text{ when } \lambda \neq 0 \\
b(s)&=cs \text{ when } \lambda=0.
\end{align*}
Since $q_0(t)$ is also a one parameter subgroup of $\mathbb{C}^*$, it is of the form $e^{\delta t}$ for some $\delta \in \mathbb{C}$ and we have that $R_t$ is of the form
\[
R_t(z_1,z_2) = \big( e^{\lambda t} (z_1 + c/\lambda) - c/\lambda, e^{\delta t}z_2 \big)
\]
when $\lambda \neq 0$. Now if we conjugate both the subgroups by the translation 
\[
T(z_1,z_2)=(z_1-c/\lambda,z_2)
\] then $R_t$ becomes $(z_1,z_2) \to (e^{\lambda t}z_1, e^{\delta t}z_2)$ (which is a type (4) subgroup if $\delta \neq0$ or else a type (1) subgroup) while $S_s$ remains as it is. When $\lambda =0$, $R_t(z_1,z_2) = (z_1 +ct,e^{\delta t}z_2)$ which is subgroup of the form (2).\\

\noindent \textit{Case(ii)}: Suppose $S_s(z_1,z_2)=(z_1+s,e^{bs}z_2)$ is a subgroup of the form (2)(a) i.e., $b \neq 0$. Then
\begin{align}
R_t^1(z_1+s, e^{bs} z_2) &= R_t^1(z_1,z_2) + s \label{creln21}, {\text{ and}} \\
R_t^2(z_1+s, e^{bs}z_2) &= e^{bs}R_t^2(z_1,z_2) \label{creln22}
\end{align}
for all $s,t \in \mathbb{R}$. Writing $R_t^1(z_1,z_2) = \sum p_t^k(z_1)z_2^k$ we have from (\ref{creln21}) by comparing coefficients of $z_2^k$ for all $k \geq 1$ that
\[
e^{kbs}p_t^k(z_1+s) = p_t^k(z_1).
\]
Then considering the zeros of $p$, we get that each $p_t^k$ is a function of $t$ alone. Next, considering a value of $t$ for which $p_t^k$ is non-zero, we have $e^{kbs}=1$ for all $s \in \mathbb{R}$, so $k=0$. Thus $p_t^k \equiv 0$ for all $k>0$. Now, note that $p_t^0(z_1+s) = p_t^0(z_1) +s$, so $p_t^0(z_1) =z_1 +f(t)$ and subsequently 
\[
R_t^1(z_1,z_2) = z_1 +f(t).
\]
Since $R_t$ is a one parameter group, we have $f(t+s) =f(t) +f(s)$ so $f(t)=ct$ for some $c \in \mathbb{C}$.\\
Next consider $R_t^2(z_1,z_2) = \sum q_t^k(z_1)z_2^k$, say. Then (\ref{creln22}) gives $e^{(k-1)bs}q_t^k(z_1+s) = q_t^k(z_1)$ which implies as before, since $b \neq 0$, that $q_t^k \equiv 0$ for all $k > 1$ while $q_t^1(z_1 +s)= q_t^1(z_1)$ says $q_t$ must be a function of $t$ alone, say $b(t)$ with $b(0)=1$. Thus $R_t^2(z_1,z_2)=b(t)z_2$ with $b(t)$ being a one parameter subgroup of $\mathbb{C}^*$, so $b(t)=e^{\lambda t}$ for some $\lambda \in \mathbb{C}$
and subsequently
\[
R_t(z_1,z_2)=(z_1+ct,e^{\lambda t}z_2).
\]
Note that when $c=0$, this is a subgroup of the form (1). Thus a subgroup commuting with a subgroup of the form (2)(a) must be necessarily of the form (2) or (1).

Now suppose that $S_s(z_1,z_2)=(z_1+s, z_2)$ then as above we have that the $p_t^k$'s are function of $t$ alone, so that 
\[
R_t^1(z_1,z_2)=z_1 + \sum p_k(t)z_2^k = z_1 +p_t(z_2)
\]
with $p_0(z_2) \equiv 0$ while the same reasoning applied to $R_t^2$ gives $R_t^2(z_1,z_2)=q_t(z_2)$. Considering then the Jacobian of $R_t$ which must be a function of $t$ alone we have $\partial q_t/ \partial z_2=r_1(t)$, say which gives $q_t(z_1,z_2)=r_1(t)z_2 + r_2(t)$. Using the fact that $R_t$ is a one parameter subgroup we may identify $p_t,q_t,r_1(t),r_2(t)$ as follows. Up till now, we have 
\[
R_t(z_1,z_2)= \big( z_1+p_t(z_2), q_t(z_2) \big).
\]
So 
\[
R_{t+s}(z_1,z_2)=\big( z_1+p_{t+s}(z_2),q_{t+s}(z_2) \big)
\]
and 
\[
R_s(R_t(z_1,z_2))= \Big( z_1 + p_t(z_2) +p_s(q_t(z_2)), q_s(q_t(z_2)) \Big)
\]
from which we have for all $s,t \in \mathbb{R}$ that
\begin{align}
q_{t+s}(z) &= q_t \circ q_s(z)  \label{qeqn}, \text{ and} \\
p_t(z) + p_s(q_t(z)) &= p_{t+s}(z) \label{peqn}
\end{align}
From (\ref{creln22}) and (\ref{qeqn}) we have
\begin{align*}
r_1(t+s) &= r_1(s)r_1(t) \\
r_1(s)r_2(t) +r_2(s) &= r_2(t+s) 
\end{align*}
which gives rise to two cases:
\begin{itemize}
\item[(i)]
\begin{align*}
r_1(t) &=e^{\lambda t} \\
r_2(t) &=(c/\lambda)(e^{\lambda t}-1)
\end{align*}
for some $\lambda \in \mathbb{C}^*$ or we have the possibility 
\item[(ii)]
\begin{align*}
r_1(t) &\equiv 1 \\
r_2(t) &=ct
\end{align*}
\end{itemize}
We tackle case (i) first. In this case $q_t(z_2)=e^{\lambda t}(z_2+c/\lambda) - c/\lambda$. Let $\tilde{p}_t(z_2) = p_t(z_2 - c/\lambda)$. Then we have from (\ref{peqn}) that 
\[
p_{t+s}(z_2 - c/\lambda) =p_t(z_2 - c/\lambda) + p_s(q_t(z_2 - c/\lambda))
\]
while $p_s(q_t(z_2 - c/\lambda))= p_s(e^{\lambda t}z_2 - c/\lambda)=\tilde{p}(e^{\lambda t} z_2)$. Therefore
\begin{equation}\label{ptileqn}
\tilde{p}_{t+s}(z_2) - \tilde{p}_t(z_2) = \tilde{p}_s(e^{\lambda t} z_2)
\end{equation}
Now suppose $\tilde{p}_t(z_2) = \sum a_j(t)z_2^j$. Noting that $\tilde{p}_0(z_2)=p_0(z_2 -c/\lambda) \equiv 0$ gives $a_j(0)=0$ while translating (\ref{ptileqn}) about $\tilde{p}$ into equations for its coefficients gives 
\[
 a_j(t+s) -a_j(t) =e^{\lambda jt}a_j(s).
\] 
Dividing by $s$ on both sides and taking limits we have $a_j'(t)=a_j'(0)e^{\lambda jt}$ showing that $a_j(t)$ are of the form $(\mu_j/j\lambda)(e^{j \lambda t} -1)$ for all $j \geq 1$ and $a_0(t) =c't$. Thus, the automorphism $R_t$ is identified to be of the form
\begin{align*}
R_t(z_1,z_2) &= \Big( z_1 + \sum (\mu_j/j \lambda)(e^{j \lambda t} -1)(z_2 + c/\lambda)^j +c't, e^{\lambda t}(z_2 + c/\lambda) -c/\lambda \Big)\\
&=\Big( z_1+\sum \nu_j\big( e^{\lambda t}(z_2+a) \big)^j - \sum \nu_j(z_2+a)^j, e^{\lambda t}(z_2+a)-a \Big)
\end{align*}
where $a=c/\lambda$ and $\nu_j=\mu/j \lambda$. If we denote the polynomial $\sum \nu_j z_2^j$ by $q(z_2)$, then first conjugating by a translation $T_a(z_1,z_2) = (z_1,z_2-a)$ transforms $R_t$ into
\[
(z_1,z_2) \to \Big(z_1 + q(e^{\lambda t}z_2) - q(z_2) + c't, e^{\lambda t}z_2 \Big)
\]
which in turn transforms via the change of variables given by the elementary map 
\[
E_q(z_1,z_2) = (z_1+q(z_2),z_2)
\]
into $(z_1,z_2) \to (z_1+c't, e^{\lambda t}z_2)$ which is of the form (2) or (1) depending on whether $c' \neq 0$ or not. Observing that both the foregoing change of variables leaves $S_s$ intact, we infer that a simultaneous conjugation of $S_s$ and $R_t$ into their normal forms has been achieved.\\

\noindent Next consider the sub-case (ii), in which case $R_t$ is of the form
\[
R_t(z_1,z_2)=(z_1+p_t(z_2), z_2 +ct).
\]
Equation (\ref{peqn}) then reads $p_t(z) + p_s(z+ct) =p_{t+s}(z)$. If $c=0$ then $p_t(z)=p(z)t$ in which case $R_t(z_1,z_2) =(z_1 +p(z_2)t, z_2)$ is already in normal form. Feeding this back into the commutation equation (\ref{creln21}), shows that $p$ must be constant but in that case $R_t$ is no different a subgroup from $S_s$. In the other case $c \neq 0$, we have 
\begin{equation} \label{peqn2}
p_t(z) + p_s(z+ct) =p_{t+s}(z).
\end{equation}
Now rewrite this as 
\[
p(t+s,z) -p(t,z) = p(s,z+ct) -p(0,z+ct).
\]
Dividing by $s$ and taking limits we have 
\[
\partial p/ \partial t(t,z) = \partial p/ \partial t(0,z+ct).
\]
Observe that the right hand side of this is a polynomial in $(z+ct)$ and so will be its anti-derivative. Therefore $p(t,z) =q(z+ct) -C(z)$ and using the fact that at $t=0$, $p(0,z) \equiv 0$ we have $C(z)=q(z)$. So $p(t,z) = q(z+ct) - q(z)$ and feeding this back into  (\ref{peqn2}) this gives
\[
 \big( q(z+ct) - q(z) \big) + \big( q(z+ct+cs) - q(z) \big) =q(z+c(t+s)) - q(z)
\]
which gives $q(z+ct)=q(z)$ i.e., $q$ is constant. Therefore, $R_t(z_1,z_2)=(z_1,z_2+ct)$ which is of the form (2)(b).\\
 
\noindent \textit{Case (iii)}: $S_s(z_1,z_2)=(z_1, z_2 +p(z_1)s)$. The commutation relations now read
\begin{align}
R_t^1(z_1,z_2 +p(z_1)s) &= R_t^1(z_1,z_2) \label{creln31}, \text{ and} \\
R_t^2(z_1,z_2+p(z_1)s) &= R_t^2(z_1,z_2) + p \big( R_t^1(z_1,z_2) \big)s \label{creln32}
\end{align} 
Replacing $z_2$ by $p(z_1)z_2$ in (\ref{creln31}) gives
\[
R_t^1(z_1,p(z_1)(z_2+s))=R_t^1(z_1,p(z_1)z_2),
\]
i.e., 
\[
r_t(z_1,z_2)=R_t^1(z_1,p(z_1)z_2)
\]
satisfies $r_t(z_1,z_2+s)=r_t(z_1,z_2)$ so $r_t$ is independent of $z_2$ and consequently $R_t^1$ is independent of $z_2$ as well. Therefore $R_t^1 \in \mathbb{C}[z_1]$. Suppose $R_t^1(z_1,z_2) = \sum a_j(t)z_1^j$ then a comparison of degrees in $R^1_{t+s}(z_1)=R^1_t(R^1_s(z_1))$ shows that $d=1$. So, $R_t^1(z)=a(t)z_1 + b(t)$ with $a(t),b(t)$ satisfying 
\[
a(t+s)z_1 + b(t+s)=a(s) \big(a(t)z_1 + b(t) \big) +b(s)
\]
which breaks into
\begin{align*}
a(t+s)&=a(t)a(s), \text{ and}\\
b(t+s)&=a(s)b(t) + b(s)
\end{align*}
which as before gives rise to two cases: \\

\noindent \textit{Sub-case(i)}:
\begin{equation}
R_t^1(z_1,z_2)= e^{\lambda t}z_1 + (c_1/\lambda)(e^{\lambda t} -1) \;\; \text{for some } \lambda \in \mathbb{C}^*, c_1 \in \mathbb{C}, \text{ or}
\end{equation}
\textit{Sub-case(ii)}:
\begin{equation}
R_t^1(z_1,z_2) = z_1 +c_1t.
\end{equation}
Now since $\deg p \geq 1$, it has a root, say $z_1^0$. Then (\ref{creln32}) gives 
\[
R_t^2(z_1^0,z_2)=R_t^2(z_1^0,z_2) +  p(R_t^1(z_1^0))s
\]
which implies $p(R_t^1(z_1^0))=0$ for all $t \in \mathbb{R}$. Hence 
\begin{equation} \label{Rconst}
R_t^1(z_1^0) \equiv z_1^0
\end{equation} Let's consider sub-case (i) now: we then have $e^{\lambda t}z_1^0 + (c_1/\lambda)(e^{\lambda t} -1) \equiv z_1^0$ which gives $z_1^0 = -c_1/\lambda$ and subsequently that $p(z_1)=(z_1 +c_1/\lambda)^d$ -- recall that $p$ is monic. A consideration of the derivative of $R_t(z) \in {\rm Aut}(\mathbb{C}^2)$ shows that $\partial R_t^2/\partial z_2$ must be independent of $z_2$ and so $R_t^2$ must be of the form 
\[
R_t^2(z_1,z_2)=c(t)z_2+d(t)+f(t,z_1).
\]
Using (\ref{creln32}) again, we have 
\[
c(t)z_2 + c(t)p(z_1)s + d(t) + f(t,z_1) = c(t)z_2 + d(t) + f(t,z_1) + p(R_t^1(z_1))s
\]
which gives $p(R_t^1(z))=c(t)p(z_1)$, i.e.,
\[
e^{\lambda dt}(z_1+c/\lambda)^d = c(t)(z+c_1/\lambda)^d
\]
so $c(t)=e^{\lambda dt}$. The constraint on these coefficients coming from the group law of $R_t$ is 
\[
c(t)d(s) + d(t)=d(t+s)
\]
giving $d(t)=(c_2/{d \lambda})(e^{d \lambda t} -1)$ for some $c_2 \in \mathbb{C}$. We use once again the group law for $R_t$ which is now in the form
\[
R_t(z_1,z_2) = \Big( e^{\lambda t}(z_1 + a)-a, e^{d \lambda t} (z_2+b) -b +f_t(z_1) \Big)
\]
-- where $b=c_2/\lambda$ and $a=c_1/\lambda$ -- so that 
\[
R_s \circ R_t(z_1,z_2) = \Big( e^{\lambda(t+s)}(z_1+a)-a,e^{d \lambda s}\big(e^{d \lambda t}(z_2+b)-b+f_t(z_1)+b \big)-b +f_s \big( e^{\lambda t}(z_1+a)-a \big) \Big)
\]
which must equal $R_{s+t}$. This equation simplifies to 
\[
f_{t+s}(z_1)=e^{d \lambda s}f_t(z_1) + f_s(e^{\lambda t}(z_1+a)-a)
\]
which gives for $g_t(z_1)=f_t(z_1-a)$ that
\[
g_{t+s}(z_1)=e^{d \lambda s}g_t(z_1) + g_s(e^{\lambda t} z_1)
\]
Expanding $g_t(z_1) =\sum a_j(t)z_1^j$ and converting the above equation to those about the $a_j$'s we have
\begin{equation} \label{aeqn}
a_j(t+s) = e^{d \lambda s}a_j(t) + a_j(s)e^{\lambda jt}.
\end{equation}
Rewriting this as 
\[
\big( a_j(t+s) - a_j(t) \big)/s = a_j(t) (e^{d \lambda s} -1)/s + (a_j(s)/s) e^{\lambda jt}
\]
and noting that $a_j(0)=0$, this gives $a_j'(t) - d \lambda a_j(t) = a_j'(0)e^{\lambda jt}$
whose integrating factor is $e^{-d \lambda t}$. We therefore have 
\begin{align}
a_j(t) &=e^{d \lambda t} \Big( \mu_j/\lambda (j-d) e^{\lambda(j-d)t} + \mu_j/ \lambda(j-d) \Big) \;\; \text{ when } j \neq d, \text{ and} \\
a_d(t) &= \mu_d e^{d \lambda t} t 
\end{align}
Next, plugging in the above expression for $a_j$ into (\ref{aeqn}) we have 
\[
C_je^{j \lambda (t+s)} - C_j = C_j e^{\lambda(ds+jt)} - C_je^{d \lambda s} +C_je^{j \lambda (t+s)} - C_j e^{j \lambda t}
\]
which simplifies to $C_j(e^{\lambda ds}-1)(e^{\lambda jt}-1) =0$ for all $j \neq d$ and then it can be easily seen that $a_j \equiv 0$ for all $0 \leq j <d$. Thus $g_t(z_1)=\mu_d t e^{d \lambda t} z_1^d$ and subsequently $f_t(z_1)=\mu_d t e^{d \lambda t} (z_1+a)^d$ and thereafter 
\[
R_t(z_1,z_2) = \big( e^{\lambda t}(z_1+a) -a, e^{d \lambda t} (z_2 +b + \mu_dt(z_1+a)^d) -b \big)
\]
Now if we conjugate by the translation $T_t(z_1,z_2)=(z_1-a,z_2-b)$ then $R_t$ in the new coordinates looks like 
\[
(z_1,z_2) \to (e^{\lambda t} z_1, e^{d \lambda t}(z_2+t z_1^d) ),
\]
a subgroup of the form (5) while the fact that $S_s$ has admitted such an $R_t$ to commute with it forces $S_s$ to be of the form $(z_1,z_2) \to (z_1, z_2 + (z_1+a)^ds)$ a very special sub-case of the form (3). Moreover as we see in this case too, there is a single change of variables to simultaneously normalize both the subgroups.\\
Now let's take up the pending sub-case (ii): $R_t^1(z_1,z_2)=z_1+ct$. We begin again by looking at (\ref{Rconst}), which gives this time that $c=0$, so $R_t^1(z_1,z_2) \equiv z_1$. Then (\ref{creln32}) reads
\[
R_t^2(z_1,z_2+p(z_1)s) = R_t^2(z_1,z_2) + p(z_1)s
\]
Differentiating with respect to $z_2$ gives
\begin{equation}\label{R2nderiv}
\partial R_t^2/\partial z_2(z_1,z_2+p(z_1)s) = \partial R_t^2/\partial z_2(z_1,z_2), \text{ and}
\end{equation}
differentiating with respect to $s$ gives
\[
\partial R_t^2/\partial z_2(z_1,z_2+p(z_1)s) p(z_1) = p(z_1).
\]
So $\partial R_t^2(z_1,z_2+p(z_1)s) \equiv 1$ and then by (\ref{R2nderiv}), we get $\partial R_t^2/\partial z_2 \equiv 1$. Integrating this gives \[
R_t^2(z_1,z_2)=z_2+\tilde{g}(z_1,t). 
\]
Using the group law for $R_t$ we have
\[
\tilde{g}(z,t+s) = \tilde{g}(z,t) + \tilde{g}(z,s)
\]
from which $\tilde{g}(z_1,t)=tq(z_1)$ for some $q(z_1) \in \mathbb{C}[z_1]$. Thus the change of variables that normalises $S_s$ also serves to put $R_t$ in its normal form (3) or a one parameter family of translations.\\

\noindent \textit{Case (iv)}:
Now let $S_t(z_1,z_2)=\big( e^{at}z_1,e^{adt}(z_2+tz_1^d) \big)$ where $a \neq 0$. Then
\begin{align}
R_s^1(e^{at}z_1,e^{adt}(z_2+tz_1^d)) &= e^{at}R_s^1(z_1,z_2) \label{creln41}, \text{ and}\\
R_s^2(e^{at}z_1,e^{adt}(z_2+tz_1^d)) &= e^{adt} \big( R_s^2(z_1,z_2) + t (R_s^1(z_1,z_2))^d \big) \label{creln42}.
\end{align}
Put $z_2=0$ in (\ref{creln41}) to get $R_s^1(e^{at}z_1,te^{adt}z_1^d) = e^{at}R_s^1(z_1,0)$. Split $R_s^1$ into components depending precisely on the variables as indicated: 
\[
R_s^1(z_1,z_2) = R_s^{11}(z_1) + R_s^{112}(z_1,z_2) + R_s^{12}(z_2) + \phi(s),
\]
-- where, for instance $R_s^{112}$ denotes the sum of all those monomials in $R_s^1$ depending on $z_1$ and $z_2$ while $R_s^{12}$ is the sum of all those monomials in $R_s^1$ depending on $z_2$ alone. Noting that $R_s^1(z_1,0)=R_s^{11}(z_1) + \phi(s)$ we have
\[
R_s^{11}(e^{at}z_1) + R_s^{112}(e^{at}z_1,te^{adt}z_1^d) +R_s^{12}(te^{at}z_1^d) + \phi(s) = e^{at}\big( R_s^{11}(z_1,0) + \phi(s) \big).
\]
Every monomial in the middle two summands on the left has its coefficient divisible by $t$ and so must cancel each other. This gives 
\begin{equation}
R_s^{11}(e^{at}z_1) + \phi(s) = e^{at}R_s^{11}(z_1) + e^{at}\phi(s).
\end{equation}
As $a \neq 0$, $\phi \equiv 0$ and $R_s^{11}$ is linear: $R_s^{11}=a(s)z_1$, say. So
\begin{equation} \label{R1brkup}
R_s^1(z_1,z_2) =a(s)z_1 + R_s^{112}(z_1,z_2) + R_s^{12}(z_2).
\end{equation}
Put $z_1=0$ in (\ref{creln42}) to get
\[
R_s^2(0,e^{adt}z_2) = e^{adt}R_s^2(0,z_2) + te^{adt}(R_s^1(0,z_2))^d
\]
which gives $R_s^1(0,z_2) =0$, so $R_s^{12}(z_2) \equiv 0$ which together with (\ref{R1brkup}) show that $R_s^1$ is now in the form
\[
R_s^1(z_1,z_2) = a(s)z_1+ R_s^{112}(z_1,z_2)
\]
and 
\begin{equation} \label{R2sym}
R_s^2(0,e^{adt}z_2) = e^{adt}R_s^2(0,z_2)
\end{equation} 
which implies that the pure $z_2$ part in $R_s^2$ is linear and $R_s^2$ has no `constant' term (i.e., term depending on $t$ alone) so that we may write $R_s^2$ in a manner similar as before, as 
\[
R_s^2(z_1,z_2)=R_s^{21}(z_1) + R_s^{212}(z_1,z_2) + e^{\lambda_2 s}z_2.
\]
(so for instance, $R_s^{212}$ denotes that part of $R_s^2$ that depends both on $z_1$ and $z_2$). Now using the group law for $R_s$ we have
\[
a(s+t)z_1 + R^{112}_{s+t}(z_1,z_2) = a(t) (a(s)z_1 + R_s^{112}(z_1,z_2) ) + R_t^{112}(R_s^1,R_s^2).
\]
Comparing pure $z_1$ terms of degree $1$ on both sides we have $a(s+t) = a(s)a(t)$ and since $R_0^1(z_1,z_2) =z_1$, $a(0)=1$. So $a(t)=e^{\lambda_1t}$ for some $\lambda_1 \in \mathbb{C}$.  Now note that (\ref{creln41}) takes the form
\begin{equation}\label{Rs112}
R_s^{112}(e^{at}z_1,e^{adt}(z_2+tz_1^d)) + e^{\lambda_1 t}(e^{at}z_1) = e^{at}R_s^{112}(z_1,z_2) + e^{at}(e^{\lambda_1t}z_1).
\end{equation}
Take any monomial $m=cz_1^jz_2^k$ in $R_s^{112}$ of the highest degree in $z_2$. Then on the left this monomial transforms to $e^{jat}e^{kadt}cz_1^j(z_2+tz_1^d)^k$ which contains the monomial $m'=e^{jat}e^{kadt}cz_1^jz_2^k$. Since $S_t(z_1,z_2)$ is linear in $z_2$ and the monomial $m$ is of the highest degree in $z_2$, $m'$ does not get cancelled on the left or combine with any other monomial. So a comparison of coefficient of this monomial on both sides of the last equation (\ref{Rs112}) gives 
\[
e^{(j+kd)at} = e^{at},
\]
so $j+kd=1$ which implies that either $j=1,k=0$ or $k=d=1,j=0$ both of which mean that $m$ is not mixed and this is a contradiction. So $R_s^{112} \equiv 0$. Therefore $R_s^1(z_1,z_2) = e^{\lambda_1 s} z_1$. Next, use the group law for $R_s$ to get
\begin{multline*}
R_s^{21}(R_t^1(z_1)) + R_s^{212}(R_t^1(z),R_t^2(z_1,z_2)) + e^{\lambda_2 s} \big( e^{\lambda_1 t} z_2 + R_t^{21}(z_1) + R_t^{212}(z_1,z_2) \big) =\\ R_{s+t}^{21}(z_1) + R_{s+t}^{212}(z_1,z_2) + e^{\lambda_2(s+t)}z_2
\end{multline*}
which gives on comparing terms involving $z_1$ alone that 
\[
R_s^{21}(e^{\lambda_1 t}z_1) + e^{\lambda_2 s}R_t^{21}(z_1) + R_s^{212}(e^{\lambda_1 t} z_1, R_s^{21}(z_1)\big) = R_{s+t}^{21}(z_1).
\]
Now note that the degree of the terms coming from the third summand on the left are strictly bigger than that of $R_s^{21}$. Therefore $R_s^{212} \equiv 0$  and 
\begin{equation}\label{Rs21eqn}
R_s^{21}(e^{\lambda _1 t} z_1) + e^{\lambda_2 s} R_t^{21}(z_1) = R_{s+t}^{21}(z_1).
\end{equation}
Now, put $z_2 =0$ in (\ref{creln42}) and differentiate with respect to $t$ to get
\begin{multline*}
ae^{at}z_1\partial R_s^2/\partial z_2(e^{at}z_1,e^{adt}tz_1^d) + (e^{adt}z_1^d + adte^{adt}z_1^d)\partial R_s^2/\partial z_2 (e^{at}z_1,e^{adt}z_1^d) = \\ ade^{adt}R_s^2(z_1,0) + (R_s^1(z_1,0))^d(e^{adt}+ ade^{adt}t).
\end{multline*}
Put $t=0$ to get
\begin{equation} \label{Rs21}
az_1 \partial R_s^{21}/\partial z_1 (z) + e^{\lambda_2 s}z_1^d = ad (R_s^{21}(z_1)) + e^{\lambda_1 ds}z_1^d
\end{equation}
where $R_s^{21}$ denotes the part of $R_s^2$ involving $z_1$ terms alone i.e., $R_s^{21}(z_1)=R_s^2(z_1,0)$ giving $\partial R_s^2/\partial z_1 = dR_s^{21}/\partial z_2$. Expand $R_s^{21}(z_1) = \sum c_j(s)z_1^j$ and recall that by (\ref{R2sym}) and $R_s^2(0,z_2)=e^{\lambda_2 s}z_2$ for some $\lambda_2 \in \mathbb{C}$. Then
\[
\partial R_s^{21}/\partial z_1(z_1) =\sum jc_j(s)z_1^{j-1}
\]
So (\ref{Rs21}) reads 
\[
\sum ac_j(s)(j-d)z_1^j = (e^{\lambda_2 s} - e^{\lambda_1 ds}) z_1^d
\]
which gives $c_j(s)=0$ for all $j \neq d$ as $a \neq 0$ and $e^{\lambda_2 s} = e^{\lambda_1 ds}$, so $\lambda_2= d \lambda_1$.
Next, getting back to (\ref{Rs21eqn}) we have for the coefficient $c_d$ that 
 \[
 c_d(s)e^{d \lambda_1 t} + e^{d \lambda_1 s}c_d(t) = c_d(s+t)
\]
Rewriting this as
\[
(c_d(t+s) - c_d(t) )/s = c_d(t)\big( e^{d \lambda_1 s} -1 \big)/s + e^{d \lambda_1 t}(c_d(s)/s)
\]
gives $c_d'(t) - (d\lambda)c_d(t) = c_d'(0)e^{d \lambda t}$. So $c_d(t) = e^{d \lambda t}(\mu_1t + \mu_2)$ where $\mu_1 =c_d'(0)$ and $\mu_2=c_d(0)$. The functional equation for $c_d(t)$ also gives $c_d(0)=0=\mu_2$. So $c_d(t)=\mu_1te^{d \lambda t}$. Thus $R_s^2(z_1,z_2)=e^{d \lambda_1 s}z_2 + \mu_1se^{d \lambda s}z_1^d$, i.e.,
\[
R_s(z_1,z_2)=(e^{\lambda_1 s}z_1, e^{d \lambda_1 s}(z_2 + \mu_1 s z_1^d))
\]
which is of the form (5) and of the same degree $d$ as $S_s$, unless $\lambda_1=0$ in which case it is in the form (3).\\

\noindent \textit{Case(v)}: The remaining case is when $S_s$ is in the form (4). Since we now know the possible `commuting match' for each of the other subgroups, it is enough to consider the case when $R_t$ is also of type (4). The commutation relations read
\begin{align}
e^{as}R_t^1(z_1,z_2) = R_t^1(e^{as}z_1,e^{bs}z_2) \label{creln51}, \text{ and} \\
e^{bs}R_t^2(z_1,z_2) = R_t^2(e^{as}z_1,e^{bs}z_2) \label{creln52}.
\end{align}
which simply mean that $R_t$ is a weighted homogeneous map with respect to the weight $(a,b)$ with the weight of the components $R_t^1,R_t^2$ also being $a$ and $b$ respectively. More explicitly, writing $R_t^1(z_1,z_2)=\sum\limits_{jk} a_{jk}(t)z_1^jz_2^k$ we have by (\ref{creln51}) that 
\[
e^{as}a_{jk}(t)= e^{(ja+kb)s} a_{jk}(t)
\]
for all $t \in \mathbb{R}$.
So, if $(j,k)$ is such that $a_{jk}(t) \neq 0$ we have 
\begin{equation} \label{indexeqn}
(j-1)a + kb=0
\end{equation}
Hence, $k=0$ if and only if $j=1$. Thus if $a_{01}(t)=0$, there is no linear term in the second summand in
\[
R_t^1(z_1,z_2)=a_{10}(t)z_1 + z_2 \Big( \sum\limits_{j,k \geq 0} a_{jk}(t)z_1^jz_2^k \Big)
\]
which is a weighted homogeneous polynomial with respect to the weight $W_S = (a,b)$ of weight $a$. Similarly from (\ref{creln52}) we have 
\[
R_t^2(z_1,z_2) = b_{01}(t)z_2 + z_1 \Big( \sum\limits_{j,k \geq 0} b_{jk}(t)z_1^j z_2^k \Big)
\]
which is also homogeneous with respect to $W_S$ of weight $b$. Note that $R_t(0)=0$. Now suppose 
\begin{equation} \label{Rdefn2}
R_t = H \circ Q_t \circ H^{-1}
\end{equation} 
where $Q_t(z_1,z_2) = (e^{ct}z_1, e^{dt} z_2)$. Noting that $R_t(0)=0$ we have $H \circ Q_t \circ H^{-1}(0)=0$. Differentiating this with respect to $t$ gives
\[
DH \Big( Q_t(\zeta) \Big)\big( ce^{ct}\zeta_1,de^{dt} \zeta_2)=0
\]
where $\zeta=H^{-1}(0)$. Further putting $t=0$ gives $DH(\zeta)(\zeta)=0$ which by the invertibility of $DH$ implies that the origin is fixed by $H$ and by (\ref{Rdefn2}),
\[
DH(0) \circ Q_t(0) \circ DH^{-1}(0) = DR_t(0).
\]
Suppose 
\begin{equation*}
DH(0) =  \Big( \begin{matrix}
A & B \\ C & D
\end{matrix} \Big)  \; \text{ and }\;
DR_t(0)=  \Big( \begin{matrix}
a_{10}(t) & a_{01}(t) \\ b_{01}(t) & b_{01}(t)
\end{matrix} \Big)
\end{equation*}
and $\delta =AD-BC = {\rm \det}DH(0)$. Then,  
$DH^{-1}(0)= 1/\delta \bigl( \begin{smallmatrix}
D & -B \\ -C & A
\end{smallmatrix} \bigr)$ and 
\begin{equation} \label{matrix1}
\Big(
\begin{matrix}
A & B \\ C & D
\end{matrix}
\Big)
\Big(
\begin{matrix}
e^{ct} & {} \\ {} & e^{dt}
\end{matrix}
\Big)
\Big(
\begin{matrix}
D & -B \\ -C & A
\end{matrix}
\Big)
= \delta
\Big(
\begin{matrix}
a_{10}(t) & a_{01}(t) \\ b_{10}(t) & b_{01}(t)
\end{matrix}
\Big)
\end{equation}
which gives
\begin{equation} \label{matrix}
\Big(
\begin{matrix}
AD e^{ct} - BCe^{dt} & AB(e^{ct} - e^{dt}) \\ CD(e^{ct}- e^{dt}) & ADe^{ct} - BC e^{dt}
\end{matrix}
\Big)
= 
\Big(
\begin{matrix}
\delta a_{10}(t) & \delta a_{01}(t) \\ \delta b_{10} & \delta b_{01}(t)
\end{matrix}
\Big).
\end{equation}
{\it Sub-case (a)}: $a \neq b$. Then since $R_t^1$ is weighted homogeneous of weight $a$, $a_{01} \equiv 0$, similarly $b_{10} \equiv 0$ and the matrix on the right of (\ref{matrix}) is diagonal. Looking at the off-diagonal entries in the last matrix equality gives $AB=0=CD$ provided $c \neq d$ and in such a case $AD-BC \neq 0$ tells that the matrix for $DH(0)$ is either diagonal or anti-diagonal. Let us assume then that $c \neq d$ and by conjugation of $H$ with the flip $F(z_1,z_2)=(z_2,z_1)$ if necessary that
\begin{align} \label{Heqn}
H_1(z_1,z_2) &= Az_1 + \text{ higher degree terms}, \text{ and} \nonumber \\
H_2(z_1,z_2) &= Bz_2 + \text{ higher degree terms},
\end{align}  
i.e.,  $B=C=0$ and $\delta = AD$. By (\ref{matrix}), $a_{10}(t)= e^{ct}$, $b_{10}(t)=e^{dt}$ and the similarity of matrices in (\ref{matrix1}) shows that $\delta=1$ and by the same equation, all these three conclusions are valid even when $c=d$ (we look at the diagonal entries in (\ref{matrix}) and so we free ourselves of the assumption $c \neq d$). Notice that the higher degree terms in the expansion of $H$ in (\ref{Heqn}), may well be of weight lower than those of the variables when the weights assigned to the variables are real numbers. However we shall now show that the lowest weights component of $H$ must be a weighted homogeneous mapping either with respect to $W_S=(a,b)$ with the weights of the components being $a$ and $b$ or with respect to $W_R=(c,d)$ with the weights of the components being $c$ and $d$. To this end recall (\ref{Rdefn2}) and rewrite it as
\begin{align}
a_{10}(t)(H_1(z_1,z_2)) + A_t(H_1,H_2) &= H_1(e^{ct}z_1,e^{dt}z_2) \label{Aeqn1}, \text{ and} \\
b_{01}(t)(H_2(z_1,z_2)) + B_t(H_1,H_2) &= H_2(e^{ct}z_1,e^{dt}z_2) \label{Beqn1},
\end{align}
where 
\begin{align*}
A_t(z_1,z_2) &=\sum\limits_{(j,k) \in S_1} a_{jk}(t) z_1^j z_2^k, \text{ and} \\
B_t(z_1,z_2) &= \sum\limits_{(k,j) \in S_2} b_{jk}(t) z_1^j z_2^k.
\end{align*}
for some finite subsets $S_1,S_2$ of $S=\{ (l,m) \; : \; l \geq 0, l\neq1, m\geq1\}$ over which the sums in the definition of $R_t$ run, i.e.,
\begin{align} 
(j,k) &\in S_1 \; \text{ only if }\;(j-1)a + kb=0 \label{S1} \\
(j,k) &\in S_2 \; \text{ only if} \; ja + (k-1)b=0 \nonumber
\end{align}
Since $H$ is an open map, note that $A_t \circ H \equiv 0$ if and only if $A_t \equiv 0$.\\
First we deal with the case $A_t \not \equiv 0$. Then $A_t \circ H \not\equiv 0$; rewrite (\ref{Aeqn1}) as
\begin{equation} \label{wtcmprsn}
A_t(H_1(z), H_2(z)) = H_1(e^{ct}z_1,e^{dt}z_2) - e^{ct}H_1(z_1,z_2)
\end{equation}
Contemplate a weighted homogeneous expansion in (\ref{Aeqn1}) with respect to $(\Im a, \Im b)$ -- recall by (\ref{S1}) that the elements in $S_1$ satisfy the same equation with $a,b$ replaced by their imaginary or real parts. If $H^{(l_1)},H^{(l_2)}$ (resp. $l_1,l_2$) are the lowest weight components (resp. lowest weights) in the weighted homogeneous expansion of $H_1$ and $H_2$ respectively, then note that the right hand side in (\ref{wtcmprsn}) is weighted homogeneous of weight $l_1$ while for the left, the lowest weight component of $A_t \circ H$ comes from $A_t \circ H^{(L)}$ but apriori need not equal it -- note that the lowest weight part in the polynomial $a_{jk}(t)H_1^jH_2^k$ is indeed $a_{jk}(t)(H_1^{(l_1)})^j(H_2^{(l_2)})^k$ and it is (weighted homogeneous) of weight $jl_1 + kl_2$, possibly bigger than $l_1$. Let us denote by 
\[
S^L_1=\{ (j,k) \in S_1 \; : \; jl_1 + kl_2=l_1 \},
\]
the set of indices which give the lowest weight in $A_t \circ H^{(L)}$. Note that this lowest weight must be non-zero by definition. So $S^L_1 \neq \phi$. Pick any $(j_0,k_0) \in S^L_1$. Then by definition $(j_0 -1,k_0)$ gives a non-trivial element in the kernel of the linear map represented by
\begin{equation}\label{matrix3}
\Big(
\begin{matrix}
a_y & b_y \\
l_1 & l_2
\end{matrix} \Big).
\end{equation}
Therefore the coloumns are proportional: $l_2/l_1=b_y/a_y$. Further note that this ratio is $-(j_0-1)/k_0$. Since every index in $S_1$ is in the kernel of the first row, it must now be in that of the second as well. This means that $S_1=S_1^L$. So in particular the ratio $-(j-1)/k$ is the same for all members of $S_1$. We also have that the lowest weight component of $A_t \circ H$ is indeed $A_t \circ H^L$ and it follows by (\ref{Aeqn1}) that 
\begin{equation} \label{HL}
R_t \circ H^{(L)} = H^{(L)} \circ Q_t.
\end{equation} 
Now by the weighted homogeneity of $H^{(L)}$, it commutes with
\begin{equation*}
{}^yS_s=
\Big(
\begin{matrix} 
e^{a_ys} & 0 \\
0 & e^{b_ys}
\end{matrix}
\Big)
\end{equation*}
upto an exponent $\mu=l_2/b_y=l_1/a_y$, i.e., $H \circ {}^y S_s = ({}^y S_s)^\mu \circ H^{(L)}$. However since $H^{(L)}$ is holomorphic, its weight is the same as its signature and so
\begin{equation} \label{Snor}
H^{(L)} \circ S_s^y= \big( S_s^y \big)^{\mu} \circ H^{(L)},
\end{equation}
where
\begin{equation}
S_s^y=
\Big(
\begin{matrix}
e^{ia_ys} & 0 \\
0 & e^{ib_ys}
\end{matrix}
\Big)
\end{equation}
as well. Since $H^{(L)}(0)=0$ we have
\[
DH^{(L)}(0)  \circ S_s^y \circ (DH^{(L)}(0))^{-1} = \big( S_s^y \big)^{\mu}.
\]
Thus $S_s^y$ and $(S_s^y)^{\mu}$ are similar and their eigenvalues must match at least as unordered pairs and two cases arise: Either $\mu=1$ or $\mu =- 1$ and $b_y= -a_y$. When $\mu=1$, $(l_1,l_2)=(a_y,b_y)$ and in the other case $\mu=-1$, $(l_1,l_2)=(-a_y,a_y)$ which is not consistent with the pair of equations at (\ref{Heqn}). Thus $\mu=1$, $H^{(L)}$ commutes with $S_s^y$ and it is now apparent from (\ref{HL}) that we have the simultaneous conjugation of both the subgroups $R_t$ and $S_s^y$ (which is already in normal form) by $H^{(L)}$ which is a weighted homogeneous transformation with respect to $(a_y,b_y)$. Now let
\begin{equation}
S_s^x=
\Big(
\begin{matrix}
e^{a_xs} & 0 \\
0 & e^{b_xs}
\end{matrix}
\Big)
\end{equation}
and recall that $b_x/a_x= b_y/a_y=-(j-1)/k$ for all $(j,k) \in S_1$ and so in particular $b_y/b_x=a_y/a_x =\alpha$, say. Then $S_s^x=({}^yS_s)^{1/\alpha}$ and consequently, $H^{(L)}$ commutes with $S_s^x$ as well. So
\[
H^{(L )} \circ S_s^x S_s^y = S_s^xS_s^y \circ H^{(L)}
\]
Since $S_s^xS_s^y=S_s$, this completes together with (\ref{HL}), the verification of the possibility of simultaneous conjugation of both the type (4) subgroups (and their real and imaginary parts) to their normal forms. Further this case, namely $A_t \not \equiv 0$ or in other words, the occurrence of a non-linear commuting conjugate of the type (4) family for $S_s$ arises only when $b/a \in \mathbb{Q}^*$. \\
We may now finish the sub-case (a) by observing that in the above argument we did not have to bother about $B_t$ being zero or not; now if $A_t$ happens to be zero, we apply the above argument to $B_t$ using (\ref{Beqn1}) if we have $B_t \not  \equiv 0$ else both $A_t, B_t \equiv0$ and it is plain from equations (\ref{Aeqn1} -\ref{Beqn1}) that $H$ is weighted homogeneous with respect to $W_Q=(c,d)$ with the weights of its components being $c$ and $d$ respectively. Owing to the weights of the components of $H$ being the same as that assigned to the variables, it follows that $H$ commutes with $Q_t$ (as in the above arguments) so that by (\ref{Rdefn2}), $Q_t=R_t$ meaning that $R_t$ was already in normal form!\\ 

{\it Sub-case(b)}: $a=b$. In this case, we have by (\ref{indexeqn}) that $j+k=1$, for all indices $(j,k)$ occurring in the definition of both the components of $R_t$, which means that $\{ R_t,S_s \}$ is a family of commuting (diagaonalizable) linear maps, which as we know can be simultaneously diagonalized. \\
\qed
\begin{rem} \label{type 4}
We observe that in the arguments above, we have not made crucial use of the hypothesis that $R_t$ was conjugate to its normal form by a {\it polynomial} automorphism. In particular in case $(v)$, we did not require that $H$ be a polynomial mapping and so it follows that two commuting one parameter subgroups, each of which are conjugate in ${\rm Aut}(\mathbb{C}^2)$ to a one parameter subgroup of $D_2(\mathbb{C})$, the group of all diagonal linear operators on $\mathbb{C}^2$, can indeed be simultaneously normalized. If it is further known that one of them is conjugate to its normal form in $GA_2(\mathbb{C})$, the other already in normal form, then the former can now be normalized via a weighted homogeneous polynomial automorphism -- this can be seen by recalling especially in the case when $H$ was non-linear that the lowest weights of the components of $H$ (while expanding with respect to $(a_y,b_y)$, say) were pinned down in the foregoing argument to be $(l_1,l_2)=(a_y,b_y)$. Noting that nothing more special about the lowest weight was used than the fact that the extremal (either lowest or maximal) weight component of $m \circ H = cH_1^jH_2^k$ for a monomial $m=cz_1^jz_2^k$ is $c( H_1^{(e_1)} )^j( H_2^{(e_2)} )^k$ where $H_1^{(e_1)}, H_2^{(e_2)}$ are the extremal weight components of $H$ we conclude that that the maximal weights $(m_1,m_2) = (a_y,b_y)$ as well and consequently the desired homogeneity of $H$. Finally, it is only a matter of replacing two tuples by $n$-tuples to see (for instance the matrix at (\ref{matrix3}) gets replaced by a $2 \times n$-matrix of rank one) that we may simultaneously conjugate two commuting one parameter groups both in $S_n(\mathbb{C})$, the conjugacy class of $D_n(\mathbb{C})$ in ${\rm Aut}(\mathbb{C}^n)$, to their diagonal forms. Finally, note that by (\ref{indexeqn}) or the pair at (\ref{S1}), that if $b/a>0$ then $a=b$, $R_t$ must be a linear group and therefore $H$ can be taken to be linear. 
\end{rem}
\noindent Let us now record all possible commuting pairs of one parameter subgroups as follows.
\begin{thm} \label{nrmlfrm2}
Suppose that $S_s,R_t$ is a pair of commuting one parameter subgroups of $GA_2(\mathbb{C})$. Then, there is a change of variables to transform both these subgroups simultaneously into one of the normal forms given by the following collection of unordered pairs,
\[
\Big\{ \{1,1\},\{1,2\}, \{1,4\}, \{2,2\},\{4,4\},  \{2b,3 \} , \{3,3\}, \{3,5\}, \{5,5\} \Big\}.
\]
In the case when the pair is conjugate to $\{3,5\}$ or $\{5,5\}$ the subgroups of the forms $(3)$ and $(5)$ herein are of the same degree.
\end{thm}

Combined with theorem \ref{dim2}, this gives the 
\begin{cor} \label{nrmlfrm3}
The admissible normal forms out of those in the above theorem, for a pair of commuting one parameter groups arising as subgroups of ${\rm Aut}(\Omega)$ for some model domain $\Omega \subset \mathbb{C}^3$ and lying in the normalizer of its canonical subgroup are given by the first five pairs therein. 
\end{cor}

\section{Models when ${\rm dim}(G)=3$ -- Proof of theorem \ref{dim3}}
It is now time to refine our reductions on the form of $P$ obtained in section $3$, when only one extra dimension to $\mathsf{N}$, the normalizer of the canonical generator, was given. Suppose now that $S_s$ and $R_t$ are two commuting one parameter subgroups of $G$ with linearly independent infinitesimal generators and lying in $\mathsf{N}$. To obtain optimal refinements, we would like to make a change of coordinates that transforms both the subgroups into their normal forms simultaneously. For instance, we know that when $'S_s,'R_t$ are each conjugate to a one parameter family of translations in the $\Re z_1$ direction say, we can pass to  equivalent domains whose defining polynomials are of the form $2\Re z_3 + P(\Im z_1,z_2)$ but when these are different subgroups, i.e., with linearly independent generators, theorem \ref{dim2} does not guarantee that both the subgroups get transformed to their normal forms; we want to know if we can change variables to make both the symmetries apparent in (the same) $P$ i.e., to recast $P$ in the form $P(z_1,z_2)=P(\Im z_1,\Im z_2)$. Corollary (\ref{nrmlfrm3}) assures this for all the cases and the sought for characterization of model domains by their automorphism groups as in theorem \ref{dim3}, now follows easily -- case (iii) in that theorem being a consequence of case (iv) of theorem \ref{dim2}, according to which $P$ must be balanced with respect to both $\Delta = (\delta_1,\delta)$ and $\Gamma =(\gamma_1,\gamma_2)$, the parameters involved in the rotation subgroups $(z_1,z_2,z_3) \to (e^{i \delta_1 s} z_1, e^{i \delta_2 s}z_2, z_3)$ and $(z_1,z_2,z_3) \to (e^{i \gamma_1 t}z_1,e^{i \gamma_2 t}z_2,z_3)$, conjugate to the two different subgroups $S_s$ and $R_t$ of $G$, corresponding to its two factors of $\mathbb{S}^1$. But then if $M(z_1,z_2)$ were non-zero, it is easily seen that this forces $\Delta$ and $\Gamma$ to be proportional (with the proportionality constant being rational), contradicting the linear independence of the infinitesimal generators of $S_s$ and $R_t$.
\qed\\

We may rephrase the arguments above to draw another 
\begin{cor}
Let $\Omega_j = \big\{ z \in \mathbb{C}^3 \; : \; 2 \Re z_3 + Q_j(z_1,z_2)<0 \big\}$ where $j=1,2$ be model domains and $F: \Omega_1 \to \Omega_2$ is a biholomorphism preserving their canonical subgroups. Then $F$ must be a polynomial automorphism of $\mathbb{C}^3$.\\
Suppose further that $\Omega_1, \Omega_2$ are strictly  but non-extremely balanced domains with respect to the weights $\Delta=(\delta_1,\delta_2)$ and $\Gamma= (\gamma_1,\gamma_2)$ where $\delta_j,\gamma_j\in \mathbb{R}$, so that ${\rm Aut}(\Omega_1)$ and ${\rm Aut}(\Omega_2)$ contain the special rotation groups 
\[
S_s(z)=(e^{i\delta_1 t}z_1, e^{i\delta_2 t}z_2,z_3)
\]
and 
\[
R_t(z)=(e^{i\gamma_1 t}z_1,e^{i\gamma_2 t}z_2,z_3)
\]
respectively and that $F$ pulls back $R_t$ to a subgroup that commutes with $S_s$ and $F('0,-1)=('0,-1)$. Then $\{ \delta_1,\delta_2 \}= \mu \{ \gamma_1, \gamma_2 \}$ for some $\mu \in \mathbb{Q}^*$ and $F$ is a weighted homogeneous polynomial automorphism of $\mathbb{C}^3$ (after composing with the flip $f(z_1,z_2) = (z_2,z_1)$ if necessary) with respect to $\Delta$ and $\Gamma$. Further, the weights of the homogeneous components with respect to either of these weights in $Q_1$ are multiples of the weights of the homogeneous components in $Q_2$ by some fixed number. 
\end{cor}
\noindent To see this, the components of $F$ satisfy (\ref{1.1}) and as in sections $1$ and $2$, $'F$ is independent of $z_3$ and is a polynomial mapping while the third component is of the form 
\[
F_3(z_1,z_2,z_3) = az_3 + \phi(z_1,z_2).
\]
for some holomorphic function $\phi$. Since $\Omega_1$ surjects onto $\mathbb{C}^2$ and $F$ is algebraic, as mentioned in sections $1,2$, we have by the same arguments that $\phi({}'z) = F_3(z) - az_3$ is an entire algebraic function, hence a polynomial and $F$ a polynomial mapping. Now the jacobian of $F$ is ${\rm Jac}(F)= a{\rm Jac}('F)$ (so, $a \neq 0$) which is a function of $'z$ alone and therefore invariant under translations in the $z_3$-direction; so if ${\rm Jac}(F) \neq \phi$ then it will intersect $\Omega_1$ to contradict that $F$ is a biholomorphism. Thus, $F \in GA_3(\mathbb{C})$.\\

\noindent Moving onto the case when $\Omega_j$'s are balanced as in the assertion, we then have by remark \ref{type 4}, since $'F$ has now been ascertained to be a polynomial automorphism of $\mathbb{C}^2$, that $'F$ (after a composition with $f$ if necessary) must be weighted homogeneous either with respect to $\Delta$ with the weights of its components $F_1,F_2$ being $\delta_1,\delta_2$ or the same holds with $\Delta$ replaced by $\Gamma$.\\

\noindent Next, we have by comparing the two defining functions for $\Omega_1$ we have for some real analytic $\psi$ near the origin,
\begin{equation}
2 \Re F_3 + Q_2(F_1,F_2) = \psi(z) \Big(2 \Re z_3 + Q_1(z_1,z_2) \Big).
\end{equation} 
Comparing terms involving $z_3$ alone, we have that $\psi(z) \equiv a \in \mathbb{R}^*$ and subsequently that 
\[
2 \Re \phi(z_1,z_2) + Q_2(F_1,F_2) = a Q_1(z_1,z_2).
\]
Since $\Omega_1$ is strictly balanced, $Q_1$ cannot have pluriharmonic terms, and the same holds of $Q_2 \circ F$ as well, since $'F('0)={}'0$. Thus, $\phi$ must be an imaginary constant. Then, $-1=F_3('0,-1)=-a + \phi('0)$ shows that $a=1$ and $\phi \equiv 0$. So $F_3(z)=z_3$.  Now, by the proof in section $5$, we know that $'F^{-1} \circ {}'R_t \circ {}'F$ must be in normal form since it commutes with $S_s$ (which is already in normal form), so as in the foregoing proof, $\Delta=\mu \Gamma$ for some $\mu \in \mathbb{Q}^*$. In particular then $F$ is weighted homogeneous with respect to both $\Delta, \Gamma$ and the ratio of the weight of $F_2$ to $F_1$ is $\gamma_2/\gamma_1 = \delta_2/ \delta_1$.\\
Now we have $Q_2 \circ {}'F = Q_1$. If the weighted homogeneous expansion of $Q_2$ with respect to $\Gamma$ say, is
\[
Q_2= Q_2^{(\nu_1)} + Q_2^{(\nu_2)} + \ldots,
\]
where $0 < \nu_1 < \nu_2<\ldots$ and $F_1,F_2$ are of weights $\eta \gamma_1,\eta\gamma_2$ respectively, then note that $Q_2^{(\nu_k)} \circ F$ is weighted homogeneous of weight $ \eta\nu_k$ -- for instance if $m=cz_1^{j_1}z_2^{j_2}\ov{z}_1^{k_1}\ov{z}_2^{k_2}$ is a monomial in $Q_2^{(\nu_k)}$, then
\[
m \circ F= c F_1^{j_1}F_2^{j_2}\ov{F}_1^{k_1}\ov{F}_2^{k_2}
\]
is weighted homogeneous of weight $(j_1+k_1) \eta \gamma_1 + (j_2 + k_2)\eta \gamma_2 = \eta {\rm wt}_{\Gamma}(m) = \eta \nu_k$. Therefore, the set of all possible weights of monomials in $Q_2 \circ {}'F=Q_1$ are multiples of the weights $\nu_k$ occurring the weighted homogeneous expansion of $Q_2$ by $\eta$.\\

\noindent Finally we note by remark \ref{type 4} that the corollary holds in higher dimensions as well.
\qed

\end{document}